\documentclass[a4paper,11pt]{amsart}
\usepackage{amsmath,inputenc,euscript,amssymb,geometry}
\usepackage{graphicx}

\usepackage{soul}
\usepackage{enumerate}
\usepackage{enumitem}
\usepackage{amssymb}
\usepackage{latexsym}
\usepackage{amssymb,amsbsy,amsmath,amsfonts,amssymb,amscd}
\usepackage{hyperref}
\usepackage{xcolor}

\newtheorem{lemma}{Lemma}[section]
\newtheorem{theorem}[lemma]{Theorem}

\newtheorem{remark}[lemma]{Remark}

\usepackage{ulem}

\begin{document}
\title[]{On the energy of critical solutions\\ of the binormal flow }
\author[V. Banica]{Valeria Banica}
\address[V. Banica]{Sorbonne Universit\'e, CNRS, Universit\'e de Paris, Laboratoire Jacques-Louis Lions (LJLL), B.C. 187, 4 place Jussieu, F-75005 Paris, France, and Institut Universitaire de France (IUF)\\ Valeria.Banica@ljll.math.upmc.fr} 

\author[L. Vega]{Luis Vega}
\address[L. Vega]{Departamento de Matem\'aticas, Universidad del Pais Vasco, Aptdo. 644, 48080 Bilbao, and BCAM, Bilbao, Spain, luis.vega@ehu.es} 
\date\today

\maketitle
\begin{abstract}The binormal flow is a model for the dynamics of a vortex filament in a 3-D inviscid incompressible fluid. The flow is also related with the classical continuous Heisenberg model in ferromagnetism, and the 1-D cubic Schr\"odinger equation. We consider a class of solutions at the critical level of regularity  that generate singularities in finite time. One of our main results is to prove the existence of a natural energy associated to these solutions. This energy remains constant except at the time of the formation of the singularity when it has a jump discontinuity. 
%{\sout{At the level of the 1-D cubic Schr\"odinger equation this conservation law involve terms related to the critical invariant space $\mathcal{F}(L^\infty)$. }}
When interpreting this conservation law in the framework of fluid mechanics, it involves the amplitude of the Fourier modes of the variation of the direction of the vorticity.\end{abstract}
\section{Introduction}
In this paper we focus on qualitative and quantitative properties of singular solutions of the binormal flow. This geometric flow  describes the evolution in time of a curve $\chi (t,x)$  in $\mathbb R^3$ that is parametrized by arclength $x$, via the equation
\begin{equation}\label{BF}
\chi_t=\chi_x\wedge \chi_{xx}.
\end{equation} 
If in a 3-D fluid the vorticity is  concentrated initially along a curve, it is expected that at least in some situations the vorticity at later times is still concentrated along another curve, whose evolution is dictated by the binormal flow. This was formally derived  by Da Rios in \cite{DaR} after truncating the integral given by Biot-Savart's law (see also \cite{MuTaUkFu}, \cite{ArHa},\cite{CaTi}). A  more rigorous argument, but still under some strong assumptions, has been recently given by Jerrard and Seis in \cite{JeSe}.  \\

The binormal flow is linked to the 1-D cubic Schr\"odinger equation (NLS) in the following way. Taking the derivative in $x$ of $\chi$ we obtain that the tangent vector $T(t,x)\in\mathbb S^2$ satisfies the classical continuous Heisenberg model used in ferromagnetism
\begin{equation}\label{SM}
T_t=T\wedge T_{xx}.
\end{equation} 
Next, by considering the curvature and torsion of  $\chi(t,x)$,
%values that live at the level of the derivative of the tangent vector via Frenet frames, 
Hasimoto constructed, in the spirit of the Madelung transform, a complex valued function that satisfies the focusing 1-D cubic NLS (\cite{Ha}). Conversely, given a real function of time $a(t)$, a solution $u$ of
\begin{equation}\label{NLS}
iu_t+u_{xx}+\frac 12(|u|^2-a(t))u=0,
\end{equation} 
a point $P\in\mathbb R^3$, and an $\mathbb R^3$-orthonormal basis $(v_1, v_2, v_3)$, one can construct a solution of \eqref{BF} as follows. First define parallel frames $(T,e_1,e_2)(t,x)$ as the solutions of
\begin{equation}\label{TNtx}\left\{
\begin{split}T_x =\Re(\overline u\, N),&\quad
N_x =-u\, T,\\
T_t =\Im (\overline{u_x}\,N), &\quad
N_t =-iu_x T+\frac i2(|u|^2-a(t)) N,
\end{split}
\right.
\end{equation}
with $N=e_1+ie_2$ and initial data $(T,e_1,e_2)(t_0,x_0)=(v_1,v_2,v_3)$. It follows that $T$ constructed this way satisfies the Schr\"odinger map \eqref{SM}. 
Finally, setting 
$$\chi(t,x)=P+\int_{t_0}^t(T\wedge T_{x})(\tau,x_0)d\tau+\int_{x_0}^xT(t,s)ds,$$
we obtain that $\chi(t,x)$ satisfies the binormal flow \eqref{BF}. Note that the construction  of $\chi(t)$ is not obvious if the solution $u$ of \eqref{NLS} is not too regular. This is precisely the scenario considered in this paper.\\

Regarding \eqref{NLS} note that since $a(t)$ is real, the corresponding term can be easily removed from the equation by a change of function. From the gauge invariance in \eqref{TNtx} this will lead to the construction of the same curve. In this way we obtain the cubic NLS 
\begin{equation}\label{NLScubic}
iu_t+u_{xx}+\frac 12|u|^2u=0,
\end{equation} 
that is invariant under the scaling 
\begin{equation}\label{scaling}
u_\lambda(t,x)=\lambda u(\lambda^2 t,\lambda x).
\end{equation}
At this respect we shall say that the solutions of \eqref{BF} are critical if they are constructed from NLS solutions in a functional setting that is invariant by scaling.

Let us recall here that \eqref{NLScubic} is well-posed in $H^{s}$, for any $s\geq 0$ (\cite{GiVe},\cite{CaWe}), and that for $s<0$ the Cauchy problem is ill-posed (\cite{KPV},\cite{ChCoTa},\cite{CaKa},\cite{Ki},\cite{Oh},\cite{KiViZh},\cite{KoTa}). We recall also that well-posedness holds for data with Fourier transform in $L^p$ spaces, $p<+\infty$ (\cite{VaVe},\cite{Gr},\cite{Ch}).

It is well known that equation \eqref{NLScubic} is also invariant under Galilean transformations 
\begin{equation}\label{gal}
u_{\eta}(t,x)=e^{-i\eta^2 t+i\eta x}u(t,x-2\eta t).
\end{equation}
One of the problems with the Sobolev class is that it is not invariant under translation in Fourier space, except of course $L^2$ that is not invariant under \eqref{scaling}. As a consequence the Sobolev class is not well suited with respect to Galilean transformations. This is the reason why in our previous work \cite{BV5}  we consider initial data whose Fourier transforms are $L^2$ periodic, possibly smooth, functions. Another possibility is to measure the Fourier transform in the $L^\infty$ norm because this topology is critical for cubic NLS with respect to both symmetries \eqref{scaling} and \eqref{gal}. One of the issues that we address in this paper is the possible growth in this latter topology.\\

The binormal flow is known to develop singularities in finite time. An important class of singular solutions is the family of self-similar solutions $\{\chi_\alpha\}_{\alpha>0}$, that are determined for $t>0$ by the values of their curvature and torsion, $\frac \alpha{\sqrt{t}}$ and $\frac{x}{2t}$ respectively. The curve $\chi_\alpha(t)$ is smooth for $t>0$ and, as proved in \cite{GRV}, it has a trace at $t=0$ given by a polygonal line with just one corner of angle $\theta$, such that
\begin{equation}\label{angle}
\sin\frac\theta 2=e^{-\pi\frac{\alpha^2}{2}}.
\end{equation} 
The corresponding 1-D cubic NLS solution is 
$u_\alpha(t,x)=\alpha\frac{e^{i\frac{x^2}{4t}}}{\sqrt{t}}$, taking $a(t)=\frac {\alpha^2}t$ in \eqref{NLS}.\\

Recently, we constructed in \cite{BV5} a class of smooth solutions of the binormal flow that generate several corners in finite time. More precisely, take a polygonal line with corners located at $x=j\in\mathbb Z$ and angles $\theta_j$, and choose  $\{\alpha_j\}$ using the relation \eqref{angle}. Then, under the assumption  that some moments of the sequence  $\{\alpha_j\}$ are squared integrable,  we construct a strong smooth solution of the binormal flow for $t\neq 0$, that is a weak solution for all $t$. This solution has  the given polygonal line as trace at $t=0$. For this purpose we first construct for $t\neq0$ and $a(t)=\frac{\sum_j|\alpha_j|^2}{t}:=\frac Mt$ a unique solution of \eqref{NLS} of the form
\begin{equation}\label{ansatz}
u(t,x)=\sum_je^{-i(|\alpha_j|^2-M)\log \sqrt{t}}\tilde A_j(t)\frac{e^{i\frac{(x-j)^2}{4t}}}{\sqrt{t}}:=\sum_j A_j(t)\frac{e^{i\frac{(x-j)^2}{4t}}}{\sqrt{t}},
\end{equation}
such that $\underset{t\rightarrow 0}{\lim}\tilde A_j(t)=\alpha_j,$ and $R_j(t):=\tilde A_j(t)-\alpha_j$ satisfies
 \begin{equation}\label{decayansatzcubic}
 \sup_{0<t<1}t^{-\gamma}\|\{R_j(t)\}\|_{l^{2,s}}+t\,\|\{\partial_t R_j(t)\}\|_{l^{2,s}}<C(\{\alpha_j\}),\end{equation}
for $0<\gamma<1$ (see also \cite{Kita} for the subcubic case). Here $s\geq 3$, $\| (\beta_j)\|_{l^{2,s}}:=(\sum_j (1+|j|)^{2s}\, |\beta_j|^2)^{1/2}$, and the coefficients 
$$e^{-i|\alpha_j|^2\log \sqrt{t}}\tilde A_j(t)$$ 
solve the non-autonomous Hamiltonian system:
\begin{equation}\label{system}
i\partial_t A_k(t)=\frac{1}{ 4\pi t}\sum_{k-j_1+j_2-j_3=0}e^{-i\frac{k^2-j_1^2+j_2^2-j_3^3}{4t}}A_{j_1}(t)\overline{A_{j_2}(t)}A_{j_3}(t)-\frac{\sum_j|\alpha_j|^2}{2\pi t}A_k(t).
\end{equation}\smallskip
Moreover, the solution satisfies the mass conservation law:
\begin{equation}\label{consl2}
M=\sum_j|\alpha_j|^2=\sum_j|A_j(t)|^2.
\end{equation}
Then, given this unique solution  of \eqref{NLS}  we construct the solution of the binormal flow as explained above. This solution has as initial data the given polygonal line. We refer the reader to Theorem 1.1 and Theorem  1.4  in \cite{BV5} for the precise statements.\\

%As far as we know the aforementioned results are the first ones obtained for \eqref{BF}, \eqref{SM}, and \eqref{NLS} in a generic class of initial conditions that are at the critical level of regularity. As explained before this generic class is given by initial datum whose Fourier transform is periodic.
% (\textcolor{violet}{don't the ones with Susana represent another class ? or our perturbations of self-similar solutions ?}). Moreover, we proved in \cite{BV5} that for \eqref{NLS} and for data in that class, there is a natural definition of mass given by \eqref{consl2} that is preserved under the flow. 
Our main result in this paper is to see if there are quantities as \eqref{consl2} associated to \eqref{BF} and \eqref{SM} that are also conserved. Recall that for smooth solutions of \eqref{SM} the energy density is given by
$$c^2\,dx= |T_x|^2\,dx,$$
where $c$ stands for the curvature. As a consequence, those  solutions of  \eqref{SM} that are constructed from solutions of \eqref{NLS} which have finite $L^2$ norm will have energy that is also finite. But this is not the case for the solutions considered in this article.

It turns out that the right way of interpreting \eqref{consl2} is to look at the Fourier transform in space of  $T_x$. Then, the energy appears as a scattering energy that is preserved as long as $t\neq0$, while it has a jump at $t=0$. More concretely, we have the following result.

%A\for When interpreting the fluid mechanics kinetic energy at the level of the binormal flow, for instance by using the correspondance between the Hamilton-Poisson structures giving rise to binormal flow and Euler flow described in \cite{JeSe}, the energy of the above infinite length binormal flow solutions is infinite. In the next result we get for these solutions a finite energy framework and we point out a transfert type phenomena. 
%Moreover we observe a sudden break down of the conservation law, and that the  new quantity still makes sense at the time of singularity formation.
\begin{theorem}\label{th} Let $\chi$ be a binormal flow solution with initial data a polygonal line, as introduced above, and $T$ its tangent vector. 
We define
\begin{equation}\label{energydef}
\Xi(T(t)):=\underset{k\rightarrow\infty}{\lim}\int_k^{k+1}|\widehat{T_x}(t,\xi)|^2d\xi.
\end{equation}
For $t>0$ we have the following conservation law:
\begin{equation}\label{cons}
\Xi(T(t))=4\pi\sum_j|\alpha_j|^2.
\end{equation}
At $t=0$ when singularities are created for the binormal flow solution $\chi$ we have
\begin{equation}\label{energyt0}
\int_k^{k+1}|\widehat{T_x}(0,\xi)|^2d\xi=4\sum_j(1-e^{-\pi |\alpha_j|^2})\quad \forall k\in\mathbb Z.
\end{equation}
Therefore there is a jump discontinuity of $\Xi(T(t))$ at time $t=0$, showing an instantaneous growth for positive times at large frequencies:
\begin{equation}\label{energy0}
\Xi(T(0))=4\sum_j(1-e^{-\pi |\alpha_j|^2})<4\pi\sum_j|\alpha_j|^2=\Xi(T(t)).
\end{equation}

\end{theorem}

The proof of the theorem is based on a careful decomposition of $\widehat{T_x}(t,\xi)$ in principal terms that eventually  give $\Xi(T(t))$ and terms for which we get either a constant type upper-bound or a logarithmic type upper-bound depending on $d(4\pi \xi,\frac{\mathbb Z}{t})$, and that become negligible in the computation of $\Xi(T(t))$.

\begin{remark}
Observe that on the one hand that the quantity $\Xi(T(t))$ involves $\hat T_x(t,\xi)$ for large $\xi$, and therefore it measures the size of the amplitude of the large frequency waves of the variation of $T$. On the other hand $T$, when interpreted at  the level of fluid mechanics,  gives the direction of the vorticity. At this respect Constantin-Fefferman-Majda's criterion \cite {CFM} states that the growth in the variation of the direction of the vorticity is necessary to produce singularities in Euler equations in three dimensions. 
\end{remark}

\begin{remark}\label{Nrem}
A similar statement holds for the normal vector, namely for $t>0$
\begin{equation}\label{consN}
\Xi(N(t)):=\underset{k\rightarrow\infty}{\lim}\int_k^{k+1}|\widehat{N_x}(t,\xi)|^2d\xi=4\pi\sum_j|\alpha_j|^2,
\end{equation}
but
\begin{equation}\label{energyN0}
\Xi(\tilde N(0))=4\sum_j(1-e^{-\pi |\alpha_j|^2}),
\end{equation}
where $\tilde N(0,x)$ is the limit at $t=0$ of \footnote{the existence of $\tilde N(0,x)$ is proved in Lemmas 4.5 in \cite{BV5}.}
$$\tilde N(t,x)=e^{i\sum_{r\in\mathbb Z, r\neq x}|\alpha_r|^2\log\frac{|x-r|}{\sqrt{t}}}N(t,x).$$ 

\end{remark}

\begin{remark}
Theorem \ref{th} applies in particular to the case of self-similar solutions of the binormal flow that are generated by polygonal lines with only one corner. Moreover,  using a perturbation argument, we constructed in \cite{BV4} solutions of the binormal flow that are smooth except at one time when they generate a corner.  For these perturbed solutions we managed to show in \cite{BV4note} that
$$\underset{\xi\rightarrow \infty}{\lim}|\widehat{T_x}(t,\xi)|^2=4\pi|\alpha_0|^2,$$
and that there exists $\epsilon>0$, depending on the perturbation of the initial data with respect to the self-similar case, such that for any $\xi\in\mathbb R$
$$|\widehat{T_x}(0,\xi)|^2<4(1-e^{-\pi |\alpha_0|^2})+\epsilon.$$
In particular for small perturbations we obtain for any $t>0$
$$\Xi(T(0))<4\pi|\alpha_0|^2=\Xi(T(t)).$$
A similar statement holds for the normal vector $N(t)$.
\end{remark}

Our final result is an observation that uses Theorem \ref{th} to reinforce the conjecture done in \cite{DHV1} about the evolution of a regular planar polygon according to the binormal flow (see also \cite{GrDe}, \cite{JeSm2}, \cite{DHKV}). In that paper, and after some theoretical arguments, it is conjectured  that the evolution of a regular polygon is periodic in time, and that at rational multiples of the time period the curve is a skew polygon with the same angle between consecutive sides. In \cite{DHV1} the size of this angle is guessed from the data obtained in the numerical simulations, while in this paper we obtain it from the energy $\Xi(T(t))$.

The paper is organized as follows. In the next section we prove the asymptotic behavior in space of the tangent and modulated normal vectors, and see that this behavior is independent of time. This information allows us to prove Theorem \ref{th} in \S\ref{sectth}. Finally, in the last section we make the observation about planar regular polygons mentioned above.

\bigskip

\section{Asymptotic behavior in space of the orthonormal frame}\label{frameas} 

\begin{lemma}\label{lemmalimT}
There exist  $T^{\pm \infty}$ with $|T^{\pm \infty}|=1$ such that for all $t>0$
\begin{equation}\label{limT}T^{\pm
 \infty}=\underset{x\rightarrow\pm\infty}{\lim}T(t,x).\end{equation}
Moreover,
$$|T(t,x)-T^{\pm \infty}|\leq\frac{C(t,\{\alpha_j\})}{\langle x\rangle}, \forall x\in\mathbb R^*, \pm x>0,$$
where $\langle x\rangle =1+|x|$.
\end{lemma}
\begin{proof} We shall first prove that for fixed $t>0$ there exists a unit vector $T^{\infty}(t)$ which is the limit of $T(t,x)$ as $x$ goes to $\infty$ ; the asymptotic behavior at $-\infty$ can be treated in the same way. 

As $T_x=\Re(\overline{u} N)$ we get for $0<x_1<x_2$:
$$T(t,x_2)-T(t,x_1)=\Re \int_{x_1}^{x_2}\sum_j\overline{A_j(t)}\frac{e^{-i\frac{(x-j)^2}{4t}}}{\sqrt{t}} N(t,x) dx.$$
We perform an integration by parts using the quadratic oscillatory phase to get $\frac 1x$ decay in space:
$$T(t,x_2)-T(t,x_1)=\left[\Re \sum_j\overline{A_j(t)}e^{-i\frac{x^2}{4t}}\frac{4t}{-i2x}\frac{e^{i\frac{xj}{2t}-i\frac{j^2}{4t}}}{\sqrt{t}} N(t,x) \right]_{x_1}^{x_2}$$
$$-\Im 2\sqrt{t} \sum_j\overline{A_j(t)}e^{-i\frac{j^2}{4t}}\int_{x_1}^{x_2} e^{-i\frac{x^2}{4t}}\left(\frac{e^{i\frac{xj}{2t}}}{x} N(t,x)\right)_x dx.$$
Since $N_x=-uT$,
$$\left|T(t,x_2)-T(t,x_1)-\Im \frac{i}{\sqrt{t}} \sum_j j\overline{A_j(t)} e^{-i\frac{j^2}{4t}}\int_{x_1}^{x_2} e^{-i\frac{x^2}{4t}}\frac{e^{i\frac{xj}{2t}}}{x} N(t,x) dx\right.$$
$$\left.-\Im 2\sqrt{t} \sum_j\overline{A_j(t)} e^{-i\frac{j^2}{4t}}\int_{x_1}^{x_2} e^{-i\frac{x^2}{4t}}\frac{e^{i\frac{xj}{2t}}}{x}\sum_kA_k(t) \frac{e^{i\frac{(x-k)^2}{4t}}}{\sqrt{t}} T(t,x) dx\right|\leq C\frac{\sqrt{t}\|\{A_j(t)\}\|_{l^1}}{x_1}.$$
In the first integral we perform again an integration by parts using the quadratic phase to obtain integrability in space:
$$\left|T(t,x_2)-T(t,x_1)-2\Im \sum_{j\neq k}\overline{A_j(t)}A_k(t) e^{i\frac{k^2-j^2}{4t}}\int_{x_1}^{x_2} e^{i\frac{x(j-k)}{2t}}\frac{T(t,x)}{x} dx\right|$$
$$\leq C\left(\frac{\sqrt{t}\|\{A_j(t)\}\|_{l^1}}{x_1}+\frac{\sqrt{t}\|\{jA_j(t)\}\|_{l^1}}{x_1^2}+\frac{\|\{j^2A_j(t)\}\|_{l^1}}{x_1\sqrt{t}}+\frac{\|\{jA_j(t)\}\|_{l^1}\|\{A_j(t)\}\|_{l^1}}{x_1}\right).$$
Above we have used that the term $j=k$ cancels.
 Now we perform an integration by parts using the linear phase, even though we don't improve the decay in $x$:
$$\left|T(t,x_2)-T(t,x_1)+2\Im \sum_{j\neq k}\overline{A_j(t)}A_k(t) e^{i\frac{k^2-j^2}{4t}}\int_{x_1}^{x_2} e^{i\frac{x(j-k)}{2t}}\frac{2t}{i(j-k)}\left(\frac{T(t,x)}{x}\right)_x dx\right|$$
$$\leq C\left(\frac{\sqrt{t}\|\{A_j(t)\}\|_{l^1}}{x_1}+\frac{\sqrt{t}\|\{jA_j(t)\}\|_{l^1}}{x_1^2}+\frac{\|\{j^2A_j(t)\}\|_{l^1}}{x_1\sqrt{t}}+\frac{\|\{jA_j(t)\}\|_{l^1}\|\{A_j(t)\}\|_{l^1}}{x_1}\right).$$
In this way we can use that $T_x=\Re(\overline{u} N)$, so that a new oscillatory term with a quadratic phase appears:
$$\left|T(t,x_2)-T(t,x_1)\right.$$
$$+2\sqrt{t}\Im \sum_{j\neq k; r}\overline{A_j(t)}A_k(t)\overline{A_r(t)} \frac{e^{i\frac{-r^2+k^2-j^2}{4t}}}{i(j-k)}\int_{x_1}^{x_2} e^{-i\frac{x^2}{4t}}\frac{e^{i\frac{x(j-k+r)}{2t}}}{x}N(t,x) dx$$
$$\left.+2\sqrt{t}\Im \sum_{j\neq k; r}\overline{A_j(t)}A_k(t)A_r(t) \frac{e^{i\frac{r^2+k^2-j^2}{4t}}}{i(j-k)}\int_{x_1}^{x_2} e^{i\frac{x^2}{4t}}\frac{e^{i\frac{x(j-k-r)}{2t}}}{x}\overline{N(t,x)} dx\right|$$
$$\leq C\left(\frac{\sqrt{t}\|\{A_j(t)\}\|_{l^1}}{x_1}+\frac{\sqrt{t}\|\{jA_j(t)\}\|_{l^1}}{x_1^2}+\frac{\|\{j^2A_j(t)\}\|_{l^1}}{x_1\sqrt{t}}+\frac{\|\{jA_j(t)\}\|_{l^1}\|\{A_j(t)\}\|_{l^1}}{x_1}\right).$$
Hence,  we can perform again an integration by parts to get decay in space:
$$\left|T(t,x_2)-T(t,x_1)\right.$$
$$-4t\sqrt{t}\Re \sum_{j\neq k; r}\overline{A_j(t)}A_k(t)\overline{A_r(t)}  \frac{e^{i\frac{-r^2+k^2-j^2}{4t}}}{i(j-k)}\int_{x_1}^{x_2} e^{-i\frac{x^2}{4t}}\left(\frac{e^{i\frac{x(j-k+r)}{2t}}}{x^2}N(t,x)\right)_x dx$$
$$\left.+4t\sqrt{t}\Re \sum_{j\neq k; r}\overline{A_j(t)}A_k(t)A_r(t)  \frac{e^{i\frac{r^2+k^2-j^2}{4t}}}{i(j-k)}\int_{x_1}^{x_2} e^{i\frac{x^2}{4t}}\left(\frac{e^{i\frac{x(j-k-r)}{2t}}}{x^2}\overline{N(t,x)}\right)_x dx\right|$$
$$\leq C\left(\frac{\sqrt{t}\|\{A_j(t)\}\|_{l^1}}{x_1}+\frac{\sqrt{t}\|\{jA_j(t)\}\|_{l^1}}{x_1^2}+\frac{\|\{j^2A_j(t)\}\|_{l^1}}{x_1\sqrt{t}}+\frac{\|\{jA_j(t)\}\|_{l^1}\|\{A_j(t)\}\|_{l^1}}{x_1}+\frac{\sqrt{t}\|\{A_j(t)\}\|_{l^1}^3}{x_1^2}\right).$$
As $N_x=-uT$ and as $|T(t,x_2)-T(t,x_1)|\leq 2$ we have obtained for $0<x_1<x_2$:
$$\left|T(t,x_2)-T(t,x_1)\right|\leq \frac{C(t,\{\alpha_j\})}{\langle x_1\rangle},$$
with
\begin{equation}\label{constT}
C(t,\{\alpha_j\})=C\left(1+\sqrt{t}\|\{jA_j(t)\}\|_{l^1}+\frac{\|\{j^2A_j(t)\}\|_{l^1}}{\sqrt{t}}+\|\{jA_j(t)\}\|_{l^1}\|\{A_j(t)\}\|_{l^1}\right.
\end{equation}
$$\left.+\sqrt{t}\|\{A_j(t)\}\|_{l^1}^3+\frac{\|\{A_j(t)\}\|_{l^1}^3}{\sqrt{t}}+t\|\{A_j(t)\}\|_{l^1}^4\right).$$
By making $x_1,x_2\rightarrow\infty$ we thus obtain the existence of
\begin{equation}\label{limTt}T^{\infty}(t):=\underset{x\rightarrow\infty}{\lim}T(t,x),\end{equation}
with the desired rate of convergence of the statement.

Now we shall prove that this vector limit is independent of $t>0$. Let $0<t_1<t_2$ and $\epsilon>0$. In view of \eqref{limTt} we can choose $x_0$ such that for all $x\geq x_0$ we have
$$|T(t_1,x)-T^\infty(t_1)|+|T(t_2,x)-T^\infty(t_2)|\leq\epsilon.$$
Thus in order to get the conclusion \eqref{limT} of the Lemma, it will be enough to find $x\geq x_0$ such that 
\begin{equation}\label{limTtx}|T(t_2,x)-T(t_1,x)|\leq\epsilon.\end{equation}
To this purpose we use that $T_t=\Im(\overline{u_x} N), N_t=-iu_x T+i\left(\frac{|u|^2}2-\frac{M}{2t}\right)N$. These expressions involve a loss of $x$. However, if a quadratic oscillatory phase $e^{-i\frac{x^2}{4t}}$ is present, integrating it in time yields $\frac1{x^2}$ decay, so eventually we gain $\frac 1x$ decay with each such integration by parts:
$$T(t_2,x)-T(t_1,x)=\Im \int_{t_1}^{t_2}\sum_je^{i(|\alpha_j|^2-M)\log \sqrt{t}}\overline{\tilde A_j(t)}\frac{e^{-i\frac{(x-j)^2}{4t}}}{\sqrt{t}}(-i)\frac{x-j}{2t} N(t,x) dt$$
$$=O(\frac 1x)-2\Im  \int_{t_1}^{t_2}\sum_j e^{-i\frac{x^2}{4t}}\frac{x-j}{x^2}\left(e^{i(|\alpha_j|^2-M)\log \sqrt{t}}\overline{\tilde A_j(t)}e^{i\frac{xj}{2t}-i\frac{j^2}{4t}}\sqrt{t}N(t,x)\right)_t dt$$
$$=O(\frac 1x)+\Re  \int_{t_1}^{t_2}\sum_j e^{-i\frac{x^2}{4t}}\frac{x-j}{x}e^{i(|\alpha_j|^2-M)\log \sqrt{t}}j\overline{\tilde A_j(t)}\frac{e^{i\frac{xj}{2t}-i\frac{j^2}{4t}}}{t\sqrt{t}}N(t,x) dt$$
$$-2\Im  \int_{t_1}^{t_2}\sum_j e^{-i\frac{x^2}{4t}}\frac{x-j}{x^2}e^{i(|\alpha_j|^2-M)\log \sqrt{t}}\overline{\tilde A_j(t)}e^{i\frac{xj}{2t}-i\frac{j^2}{4t}}\sqrt{t}N_t(t,x) dt.$$
In the first integral we perform again an integration by parts from the quadratic case to get the desired $\frac 1x$ decay, while for the second integral we have to treat only the $iu_xT$ part of $N_t$:
$$T(t_2,x)-T(t_1,x)=O(\frac 1x)$$
$$+2\Im  \int_{t_1}^{t_2}\sum_{j\neq k} \frac{(x-j)(x-k)}{x^2}e^{i(|\alpha_j|^2-|\alpha_k|^2)\log \sqrt{t}}\overline{\tilde A_j(t)}\tilde A_k(t)e^{i\frac{x(j-k)}{2t}-i\frac{j^2-k^2}{4t}}\frac {T(t,x)}{t} dt.$$
Now we perform an integration by parts using the linear phase in $x$ to get:
$$T(t_2,x)-T(t_1,x)=O(\frac 1x)$$
$$+4\Re  \int_{t_1}^{t_2}\sum_{j\neq k} \frac{(x-j)(x-k)}{x^3(j-k)}e^{i\frac{x(j-k)}{2t}}\left(e^{i(|\alpha_j|^2-|\alpha_k|^2)\log \sqrt{t}}\overline{\tilde A_j(t)}\tilde A_k(t)e^{-i\frac{j^2-k^2}{4t}}tT(t,x)\right)_t dt$$
$$=O(\frac 1x)$$
$$+4\Re  \int_{t_1}^{t_2}\sum_{j\neq k} \frac{(x-j)(x-k)}{x^3(j-k)}e^{i\frac{x(j-k)}{2t}}e^{i(|\alpha_j|^2-|\alpha_k|^2)\log \sqrt{t}}\overline{\tilde A_j(t)}\tilde A_k(t)e^{-i\frac{j^2-k^2}{4t}}t\times$$
$$\times\Im \left(\sum_r e^{i(|\alpha_r|^2-M)\log \sqrt{t}}\overline{\tilde A_r(t)}\frac{e^{-i\frac{(x-r)^2}{4t}}}{\sqrt{t}}(-i)\frac{x-r}{2t}  N(t,x)\right)dt.$$
Although  we  still do not get enough decay in $x$  we have got a quadratic phase in $x$. Hence, we perform another integration by parts using it to get an extra $\frac 1x$ decay:
$$T(t_2,x)-T(t_1,x)=O(\frac 1x).$$
Therefore we can find $x$ depending on $x_0,t_1,t_2$ and $\{\alpha_j\}$ such that \eqref{limTtx} holds, and the Lemma follows.

\end{proof}

\begin{lemma}\label{lemmalimN}
There exist $N^{\pm \infty}\in \Bbb S^2+ i\Bbb S^2$, $\Bbb S^2$ denoting the unit sphere in $\mathbb R^3$, such that for all $t>0$
\begin{equation}\label{limN}N^{\pm \infty}=\underset{x\rightarrow\pm\infty}{\lim}N_M(t,x),\end{equation}
where for $x\neq 0$
$$N_M(t,x)=e^{iM\log\frac {|x|}{\sqrt{t}}}N(t,x).$$
As a consequence we also have
$$N^{\pm \infty}=\underset{x\rightarrow\pm\infty}{\lim}e^{iM\log\frac {\langle x\rangle}{\sqrt{t}}}N(t,x).$$
Moreover, we have the following rate of convergence 
$$|e^{iM\log\frac {\langle x\rangle}{\sqrt{t}}}N(t,x)-N^{\pm \infty}|\leq\frac{C(t,\{\alpha_j\})}{\langle x\rangle}, \forall x\in\mathbb R^{*\pm}.$$

\end{lemma}
\begin{proof} As done for the tangent vector, we shall first prove that for fixed $t>0$ there exists a limit vector $N^{\infty}(t)$ for $N_M(t,x)$ as $x$ goes to $\infty$ ; the asymptotic at $-\infty$ can be treated in the same way. 

As for $x>0$
$$(N_M)_x=(-uT+i\frac{M}{x}N)e^{iM\log\frac x{\sqrt{t}}},\quad T_x=\Re(\overline{u}N),$$ 
we get for $0<x_1<x_2$ by integrating by parts:
$$N_M(t,x_2)-N_M(t,x_1)$$
$$=\int_{x_1}^{x_2}\left(-\sum_jA_j(t)\frac{e^{i\frac{(x-j)^2}{4t}}}{\sqrt{t}} T(t,x)+i\frac{M}{x}N(t,x)\right)e^{iM\log\frac x{\sqrt{t}}} dx$$
$$=\left[-\sum_jA_j(t)e^{i\frac{x^2}{4t}}\frac{2t}{ix}\frac{e^{-i\frac{xj}{2t}+i\frac{j^2}{4t}  }}{\sqrt{t}} T(t,x)e^{iM\log\frac x{\sqrt{t}}}\right]_{x_1}^{x_2}$$
$$-\int_{x_1}^{x_2}2i\sqrt{t}\sum_je^{i\frac{j^2}{4t} }A_j(t)e^{i\frac{x^2}{4t}}\left(e^{-i\frac{xj}{2t} }T(t,x)\frac{e^{iM\log\frac x{\sqrt{t}}}}{x}\right)_xdx$$
$$+\int_{x_1}^{x_2}i\frac{M}{x}N(t,x)e^{iM\log\frac x{\sqrt{t}}} dx.$$
Thus
$$\left|N_M(t,x_2)-N_M(t,x_1)+\int_{x_1}^{x_2}\frac 1{\sqrt{t}}\sum_je^{i\frac{j^2}{4t} }jA_j(t)e^{i\frac{x^2}{4t}}e^{-i\frac{xj}{2t} }T(t,x)\frac{e^{iM\log\frac x{\sqrt{t}}}}{x}dx\right.$$
$$+\int_{x_1}^{x_2}2i\sum_jA_j(t)e^{i\frac{(x-j)^2}{4t}}\Re\left(\sum_k\overline{A_k(t)}e^{-i\frac{(x-k)^2}{4t}}N(t,x)\right)\frac{e^{iM\log\frac x{\sqrt{t}}}}{x}dx$$
$$\left.-\int_{x_1}^{x_2}i\frac{M}{x}N(t,x)e^{iM\log\frac x{\sqrt{t}}} dx\right|\leq C\frac{\sqrt{t}\|\{A_j(t)\}\|_{l^1}}{x_1}.$$
In the first integral we perform again an integration by parts using the quadratic phase $x^2$, and get integrability with a $\frac 1{x_1}$ decay. In the second integral we develop the real part. The diagonal $k=j$ terms of its non-conjugated part cancel with the third integral, as we have the conservation law $M=\sum_j|\alpha_j|^2=\sum_j|A_j(t)|^2$. We are left with:
$$\left|N_M(t,x_2)-N_M(t,x_1)+i\sum_{j\neq k}A_j(t)\overline{A_k(t)}e^{i\frac{j^2-k^2}{4t}}\int_{x_1}^{x_2}e^{-i\frac{x(j-k)}{2t}}N(t,x)\frac{e^{iM\log\frac x{\sqrt{t}}}}{x}dx\right.$$
$$\left.+i\sum_{j,k}A_j(t)A_k(t)e^{i\frac{j^2+k^2}{4t}}\int_{x_1}^{x_2}e^{i\frac{x^2}{2t}}e^{-i\frac{x(j+k)}{2t}}\overline{N(t,x)}\frac{e^{iM\log\frac x{\sqrt{t}}}}{x}dx\right|.$$
$$\leq C\left(\frac{\sqrt{t}\|\{A_j(t)\}\|_{l^1}}{x_1}+\frac{\sqrt{t}\|\{jA_j(t)\}\|_{l^1}}{x_1^2}+\frac{\|\{j^2A_j(t)\}\|_{l^1}}{x_1\sqrt{t}}+\frac{\|\{jA_j(t)\}\|_{l^1}\|\{A_j(t)\}\|_{l^1}}{x_1}\right)$$
In the second integral, a new integration by parts using the quadratic phase $x^2$ yields integrability with a $\frac 1{x_1}$ decay. In the first integral we integrate by parts using the linear phase $x(j-k)$:
$$\left|N_M(t,x_2)-N_M(t,x_1)+2t\sum_{j\neq k}\frac{A_j(t)\overline{A_k(t)}}{j-k}e^{i\frac{j^2-k^2}{4t}}\int_{x_1}^{x_2}e^{-i\frac{x(j-k)}{2t}}N_x(t,x)\frac{e^{iM\log\frac x{\sqrt{t}}}}{x}dx\right|$$
$$\leq C\left(\frac{\sqrt{t}\|\{A_j(t)\}\|_{l^1}}{x_1}+\frac{\sqrt{t}\|\{jA_j(t)\}\|_{l^1}}{x_1^2}+\frac{\|\{j^2A_j(t)\}\|_{l^1}}{x_1\sqrt{t}}+\frac{\|\{jA_j(t)\}\|_{l^1}\|\{A_j(t)\}\|_{l^1}}{x_1}+\frac{\sqrt{t}\|\{A_j(t)\}\|_{l^1}^3}{x_1^2}\right).$$

As $N_x(t,x)=-uT(t,x)$ contains $e^{i\frac{x^2}{4t}}$, we perform a last integration by parts using this quadratic phase to get for all $0<x_1<x_2$:
$$\left|N_M(t,x_2)-N_M(t,x_1)\right|\leq \frac{C(t,\{\alpha_j\})}{\langle x_1\rangle},$$
with the same constant $C(t,\{\alpha_j\})$ as in \eqref{constT}:
$$C(t,\{\alpha_j\})=C\left(1+\sqrt{t}\|\{jA_j(t)\}\|_{l^1}+\frac{\|\{j^2A_j(t)\}\|_{l^1}}{\sqrt{t}}+\|\{jA_j(t)\}\|_{l^1}\|\{A_j(t)\}\|_{l^1}\right.$$
$$\left.+\sqrt{t}\|\{A_j(t)\}\|_{l^1}^3+\frac{\|\{A_j(t)\}\|_{l^1}^3}{\sqrt{t}}+t\|\{A_j(t)\}\|_{l^1}^4\right).$$
It follows that we have a limit \begin{equation}\label{limNt}
N^\infty(t):=\underset{x\rightarrow\infty}{\lim}N_M(t,x),\end{equation}
with a rate of convergence in space as in the statement of the lemma.

We are thus left to show the independence on time of $N^\infty(t)$. We fix $0<t_1<t_2$ and $\epsilon>0$, choose $x_0$ such that 
$$|N_M(t_1,x)-N^\infty(t_1)|+|N_M(t_2,x)-N^\infty(t_2)|\leq\epsilon.$$
To finish the proof of the lemma, it will be enough to find $x\geq x_0$ such that 
\begin{equation}\label{limNtx}|N_M(t_2,x)-N_M(t_1,x)|\leq\epsilon.\end{equation}
As the evolution in time laws are 
$$T_t=\Im(\overline{u_x} N),\quad (N_M)_t=\left(-iu_x T+i\left(\frac{|u|^2}2-\frac{M}{2t}\right)N-i\frac M{2t}N\right)e^{iM\log\frac x{\sqrt{t}}},$$
we can write
$$N_M(t_2,x)-N_M(t_1,x)=\int_{t_1}^{t_2}\left(-i\sum_je^{-i(|\alpha_j|^2-M)\log \sqrt{t}}\tilde A_j(t)\frac{e^{i\frac{(x-j)^2}{4t}}}{\sqrt{t}}i\frac{x-j}{2t} T(t,x)\right.$$
$$\left.+i\sum_{j\neq k}e^{-i(|\alpha_j|^2-|\alpha_k|^2)\log \sqrt{t}}\tilde A_j(t)\overline{\tilde A_k(t)}\frac{e^{i\frac{j^2-k^2}{4t}-i\frac{x(j-k)}{2t}}}{2t}N-i\frac M{2t}N\right)e^{iM\log\frac x{\sqrt{t}}}dt.$$
In the first integral we perform an integration by parts using the quadratic phase, while in the second we use the linear one:
$$N_M(t_2,x)-N_M(t_1,x)=\left[\sum_je^{-i(|\alpha_j|^2-M)\log \sqrt{t}}\tilde A_j(t)\frac{e^{i\frac{(x-j)^2}{4t}}}{\sqrt{t}}(-\frac{4t^2}{ix^2})\frac{x-j}{2t} T(t,x)e^{iM\log\frac x{\sqrt{t}}}\right]_{t_1}^{t_2}$$
$$-2i\int_{t_1}^{t_2}\sum_j \frac{x-j}{x^2}e^{i\frac{x^2}{4t}}\left(e^{-i(|\alpha_j|^2-M)\log \sqrt{t}}\tilde A_j(t)e^{-i\frac{xj}{2t}+i\frac{j^2}{4t}}\sqrt{t} \,T(t,x)e^{iM\log\frac x{\sqrt{t}}}\right)_tdt$$
$$+\left[i\sum_{j\neq k}e^{-i(|\alpha_j|^2-|\alpha_k|^2)\log \sqrt{t}}\tilde A_j(t)\overline{\tilde A_k(t)}\frac{e^{i\frac{j^2-k^2}{4t}-i\frac{x(j-k)}{2t}}}{2t}\frac{2t^2}{ix(j-k)}Ne^{iM\log\frac x{\sqrt{t}}}\right]_{t_1}^{t_2}$$
$$-\int_{t_1}^{t_2}i\sum_{j\neq k}\frac{1}{x(j-k)}e^{-i\frac{x(j-k)}{2t}}\left(e^{-i(|\alpha_j|^2-|\alpha_k|^2)\log \sqrt{t}}\tilde A_j(t)\overline{\tilde A_k(t)}e^{i\frac{j^2-k^2}{4t}}t\,N e^{iM\log\frac x{\sqrt{t}}}\right)_tdt$$
$$-\int_{t_1}^{t_2}i\frac M{2t}Ne^{iM\log\frac x{\sqrt{t}}}dt$$
$$=O(\frac 1x)+\int_{t_1}^{t_2}\sum_j \frac{x-j}{x}e^{i\frac{x^2}{4t}}e^{-i(|\alpha_j|^2-M)\log \sqrt{t}}j\tilde A_j(t)e^{-i\frac{xj}{2t}+i\frac{j^2}{4t}}\frac1{t\sqrt{t}} \,T(t,x)e^{iM\log\frac x{\sqrt{t}}}dt$$
$$-2i\int_{t_1}^{t_2}\sum_j \frac{x-j}{x^2}e^{i\frac{x^2}{4t}}e^{-i(|\alpha_j|^2-M)\log \sqrt{t}}\tilde A_j(t)e^{-i\frac{xj}{2t}+i\frac{j^2}{4t}}\sqrt{t} \,T_t(t,x)e^{iM\log\frac x{\sqrt{t}}}dt$$
$$-\int_{t_1}^{t_2}i\sum_{j\neq k}\frac{1}{x(j-k)}e^{-i\frac{x(j-k)}{2t}}e^{-i(|\alpha_j|^2-|\alpha_k|^2)\log \sqrt{t}}\tilde A_j(t)\overline{\tilde A_k(t)}e^{i\frac{j^2-k^2}{4t}}t\,N_t e^{iM\log\frac x{\sqrt{t}}}dt$$
$$-\int_{t_1}^{t_2}i\frac M{2t}Ne^{iM\log\frac x{\sqrt{t}}}dt:=O(\frac 1x)+I_1+I_2+I_3+I_4.$$
In the first integral $I_1$  an integration by p arts using the quadratic phase gives us the $\frac 1x$ decay. The second integral can be rewritten as
$$I_2=O(\frac 1x)-\frac {2i}x\int_{t_1}^{t_2}\sum_j e^{i\frac{x^2}{4t}}e^{-i(|\alpha_j|^2-M)\log \sqrt{t}}\tilde A_j(t)e^{-i\frac{xj}{2t}+i\frac{j^2}{4t}}\sqrt{t} \,\Im (\overline{u_x}N(t,x))e^{iM\log\frac x{\sqrt{t}}}dt$$
$$=O(\frac 1x)+i\int_{t_1}^{t_2}\sum_{j,k} e^{i\frac{(x-j)^2-(x-k)^2}{4t}}e^{-i(|\alpha_j|^2-|\alpha_k|^2)\log \sqrt{t}}\tilde A_j(t)\overline{\tilde A_k(t)}\frac 1{2t}N(t,x)e^{iM\log\frac x{\sqrt{t}}}dt$$
$$-i\int_{t_1}^{t_2}\sum_{j,k} e^{i\frac{(x-j)^2+(x-k)^2}{4t}}e^{-i(|\alpha_j|^2+|\alpha_k|^2-2M)\log \sqrt{t}}\tilde A_j(t)\tilde A_k(t)\frac 1{2t}\overline{N(t,x)}e^{iM\log\frac x{\sqrt{t}}}dt$$
$$=O(\frac 1x)-I_4+i\int_{t_1}^{t_2}\sum_{j\neq k} e^{i\frac{x(j-k)}{2t}-i\frac{j^2-k^2}{4t}}e^{-i(|\alpha_j|^2-|\alpha_k|^2)\log \sqrt{t}}\tilde A_j(t)\overline{\tilde A_k(t)}\frac 1{2t}N(t,x)e^{iM\log\frac x{\sqrt{t}}}dt$$
$$-i\int_{t_1}^{t_2}\sum_{j,k} e^{i\frac{x^2}{2t}}e^{i\frac{-x(j+k)}{2t}+i\frac{j^2+k^2}{4t}}e^{-i(|\alpha_j|^2+|\alpha_k|^2-2M)\log \sqrt{t}}\tilde A_j(t)\tilde A_k(t)\frac 1{2t}\overline{N(t,x)}e^{iM\log\frac x{\sqrt{t}}}dt,$$
where we used the conservation law $M=\sum_j|\alpha_j|^2=\sum_j|\tilde A_j(t)|^2$. 
In the first integral we integrate by parts using the linear phase in $x$, that gives the decay $\frac 1x$ except when the derivative in time falls on $N$. This term  involves a power of $x$ but also an oscillatory term with a quadratic phase in $x$. Another integration by parts  gives eventually the decay $\frac 1x$. In the last integral a new integration by parts using the quadratic phase gives immediately the decay $\frac 1x$. Therefore
$$I_2+I_4=O(\frac 1x).$$
Finally, in $I_3$ there is a factor $\frac 1x$ and from $N_t$ we loose a power of $x$ just for the term $-u_xT$. However, this term introduces back the quadratic phase in $x$, and a new integration by parts  yields the $\frac 1x$ decay. Therefore
$$N_M(t_2,x)-N_M(t_1,x)=O(\frac 1x),$$
so \eqref{limNtx} follows. The proof of the lemma is over.

\end{proof}

\section{Proof of Theorem \ref{th}}\label{sectth}
\subsection{The result on the tangent vector} We start with the proof of the results at time $t=0$, namely \eqref{energyt0}. 
We will rely from section 4.6 in \cite{BV5} that at $t=0$ the curve is a polygonal line so that $T(0,x)$ is  piecewise constant with jumps at the integers $j\in \Bbb Z$ and that
$$T_x(0)=\sum_j(T(0,j^+)-T(0,j^-))\delta_j=\sum_j \Theta_j(A_{|\alpha_j|}^+-A_{|\alpha_j|}^-)\delta_j.$$
Here $\Theta_j$ denotes an appropriate rotation (see \cite{BV5}) and $A^\pm_{|\alpha_j|}$ are the two unit vectors representing the limits at $\pm\infty$ of the tangent of the self-similar solution $\chi_{|\alpha_j|}$. Then, we have
$$\widehat{T_x}(0,\xi)=\sum_j \Theta_j(A_{|\alpha_j|}^+-A_{|\alpha_j|}^-)e^{i2\pi j\xi}.$$
In particular $\widehat{T_x}(0,\xi)$ is periodic in $\xi$ and we get by Plancherel's theorem that for any $k$
$$\int_k^{k+1}|\widehat{T_x}(0,\xi)|^2d\xi=\sum_j |\Theta_j(A_{|\alpha_j|}^+-A_{|\alpha_j|}^-)|^2=\sum_j |A_{|\alpha_j|}^+-A_{|\alpha_j|}^-|^2.$$
Therefore calling $\theta_j$ the angle between $A_{|\alpha_j|}^+$ and $A_{|\alpha_j|}^-$ and using (3) and (4) in \cite{BV5} we have
 $$ |A_{|\alpha_j|}^+-A_{|\alpha_j|}^-|^2= 2(1-\cos \theta_j)=4(1-e^{-\pi |\alpha_j|^2}),$$ 
 so that we obtain \eqref{energyt0}, and implicitly \eqref{energy0}.\\

Now we fix $t>0$ and our purpose it to compute $\Xi(t)$ and to obtain \eqref{cons}. Let $0<\epsilon<1$. In view of \eqref{decayansatzcubic} we choose $j_\epsilon$ depending on $\epsilon, t$ and $\{\alpha_j\}$ such that
\begin{equation}\label{tailA}
\sum_{|j|\geq j_\epsilon} |A_j(t)|\leq\epsilon.
\end{equation}
In the following  $C$ will denote a generic constant dependent on $t$ and $\{\alpha_j\}$, unless it is specified othewise.

Since $T_x(t,x)=\Re(\overline{u} N)(t,x)$ we have
$$\widehat{T_x}(t,\xi)=\int_{-\infty}^\infty e^{i2\pi x\xi}\,\Re(\overline{u} N)(t,x)dx$$
$$=\int_{-\infty}^\infty e^{i2\pi x\xi}\,\Re(\sum_j\overline{A_j(t)}\frac{e^{-i\frac{(x-j)^2}{4t}}}{\sqrt{t}} N(t,x))dx.$$
We denote $\eta^+$ a smooth function vanishing on $x<-\frac 12$  and valued $1$ on $x>\frac 12$, and we denote $\eta^-=1-\eta^+$, so that
$$\widehat{T_x}(t,\xi)=\sum_\pm \int_{-\infty}^\infty e^{i2\pi x\xi}\,\Re(\sum_j\overline{A_j(t)}\frac{e^{-i\frac{(x-j)^2}{4t}}}{\sqrt{t}} N(t,x))\,\eta^\pm(x)dx.$$
With the notations from Lemma \ref{lemmalimN}, on the integral involving $\eta^\pm$ we split 
$$N(t,x)=N^{\pm \infty}  e^{-iM\log\frac  {\langle x\rangle}{\sqrt{t}}}+g_N^\pm (t,x),$$
where$$g_N^\pm (t,x):=(N(t,x)-N^{\pm \infty} e^{-iM\log\frac  {\langle x\rangle}{\sqrt{t}}}).$$
We define
\begin{equation}\label{decT}\widehat{T_x}(t,\xi)=I(t,\xi)+J(t,\xi),\end{equation}
where $I(t,\xi)$ gathers the terms in $\widehat{T_x}(t,\xi)$ corresponding to $N^{\pm \infty}$ and $J(t,\xi)$ the ones corresponding to $g_N^\pm$.
We shall start by estimating the second term $J(t,\xi)$.\\

First, we complete the squares of the phases:
$$J(t,\xi)=\frac 12\sum_\pm \int_{-\infty}^\infty e^{i2\pi x\xi}\,\sum_j\overline{A_j(t)}\frac{e^{-i\frac{(x-j)^2}{4t}}}{\sqrt{t}} g_N^\pm (t,x)\,\eta^\pm(x)dx$$
$$+\frac 12\sum_\pm \int_{-\infty}^\infty e^{i2\pi x\xi}\,\sum_jA_j(t)\frac{e^{i\frac{(x-j)^2}{4t}}}{\sqrt{t}} \overline{g_N^\pm (t,x)}\,\eta^\pm(x)dx$$
$$=\frac{e^{i4\pi^2t\xi^2}}{2\sqrt{t}}\sum_{\pm, j}e^{i2\pi j\xi}\,\overline{A_j(t)}\int_{-\infty}^\infty e^{-i\frac{(x-j-4\pi t\xi)^2}{4t}}g_N^\pm (t,x)\,\eta^\pm(x)dx$$
$$+\frac{e^{-i4\pi^2t\xi^2}}{2\sqrt{t}}\sum_{\pm, j}e^{i2\pi j\xi}\,A_j(t)\int_{-\infty}^\infty e^{i\frac{(x-j+4\pi t\xi)^2}{4t}}\overline{g_N^\pm (t,x)}\,\eta^\pm(x)dx.$$
%In particular $$\left|\widehat{T_x}(t,\xi)-\sum_{|j|\leq j_\epsilon}e^{i2\pi j\xi}\,e^{i|\alpha_j|^2\log \sqrt{t}}\overline{A_j(t)}e^{-i\frac{j^2}{4t}}e^{i4\pi^2t\xi^2}\int_{-\infty}^\infty e^{-is^2}N(t,j+4\pi t\xi-2\sqrt{t}s)ds\right.$$$$\left.+\sum_{|j|\leq j_\epsilon}e^{i2\pi j\xi}\,e^{-i|\alpha_j|^2\log \sqrt{t}}A_j(t)e^{i\frac{j^2}{4t}}e^{-i4\pi^2t\xi^2}\int_{-\infty}^\infty e^{is^2}\overline{N(t,j-4\pi t\xi+2\sqrt{t}s)}ds\right|\leq C\epsilon+\frac C{d(4\pi t\xi,\mathbb Z)}\epsilon.$$
%With the notations from Lemma \ref{lemmalimN} and Lemma \ref{lemmalimN} we split:$$\widehat{T_x}(t,\xi)=\int_{-\infty}^\infty e^{i2\pi x\xi}\,\Re(\overline{u}(t,x) (N^\infty+(N_m(t,x)-N^{+\infty}))e^{-iM\log\frac  {|x|}{\sqrt{t}}})dx$$$$=\int_{-\infty}^\infty e^{i2\pi x\xi}\,\Re(\sum_je^{i|\alpha_j|^2\log \sqrt{t}}\overline{A_j(t)}\frac{e^{-i\frac{(x-j)^2}{4t}}}{\sqrt{t}} N^\infty e^{-iM\log\frac {|x|}{\sqrt{t}}})dx$$$$+\int_{-\infty}^\infty e^{i2\pi x\xi}\,\Re(\sum_je^{i|\alpha_j|^2\log \sqrt{t}}\overline{A_j(t)}\frac{e^{-i\frac{(x-j)^2}{4t}}}{\sqrt{t}} (N_m(t,x)-N^{+\infty})e^{-iM\log\frac {|x|}{\sqrt{t}}})dx:=I_1(t,\xi)+I_2(t,\xi).$$
We split now the summation into $|j|< j_\epsilon$ and $|j|\geq j_\epsilon$, and call the corresponding terms $J^l(t,\xi)$ and $J^h(t,\xi)$.

\begin{lemma}\label{lemma_low}
There exists $\xi(\epsilon,t,\{\alpha_j\})\in\mathbb R$ such that for $\xi\geq  \xi(\epsilon,t,\{\alpha_j\})$ and $4\pi t\xi\notin\mathbb Z$ we have the bounds
$$|J^l(t,\xi)|\leq\left\{\begin{array}{c}C\epsilon, \mbox{ if }d(2\pi \xi,\frac{\mathbb Z}{2t})\geq 1,\vspace{2mm}\\C\epsilon \,|\log (d(2\pi \xi,\frac{\mathbb Z}{2t}))|, \mbox{ if }d(2\pi \xi,\frac{\mathbb Z}{2t})<  1.\end{array} \right.$$

\end{lemma}

\begin{proof}
In virtue of Lemma \ref{lemmalimN}, $g_N^\pm$ are bounded functions with 
$$|g_N^\pm (t,x)|\leq\frac{C}{\langle x\rangle}, \forall x\in\mathbb R^{*\pm},$$
so $g_N^\pm (t,x)\eta^\pm(x)$ converge to zero at both $-\infty$ and $+\infty$. Therefore we can remove from
$$J^l(t,\xi)=\frac{e^{i4\pi^2t\xi^2}}{2\sqrt{t}}\sum_{\pm, |j|<j_\epsilon}e^{i2\pi j\xi}\,\overline{A_j(t)}\int_{-\infty}^\infty e^{-i\frac{(x-j-4\pi t\xi)^2}{4t}}g_N^\pm (t,x)\eta^\pm (x)dx$$
$$+\frac{e^{-i4\pi^2t\xi^2}}{2\sqrt{t}}\sum_{\pm, |j|<j_\epsilon}e^{i2\pi j\xi}\,A_j(t)\int_{-\infty}^\infty e^{i\frac{(x-j+4\pi t\xi)^2}{4t}}\overline{g_N^\pm (t,x)}\eta^\pm(x)dx,$$
bounded pieces of the integrals in $x$ located around $j\pm 4\pi t\xi$. Indeed on these parts, since $|j|\leq j_\epsilon$, we have convergence to zero as $\xi$ goes to infinity. Therefore there exists $\xi(\epsilon,t,\{\alpha_j\})$ such that for $\xi\geq  \xi(\epsilon,t,\{\alpha_j\})$ we have
$$|J^l(t,\xi)-J_1^l(t,\xi)|\leq \epsilon,$$
where
$$J_1^l(t,\xi)=\frac{e^{i4\pi^2t\xi^2}}{2\sqrt{t}}\sum_{\pm, |j|<j_\epsilon}e^{i2\pi j\xi}\,\overline{A_j(t)}\int_{-\infty}^\infty e^{-i\frac{(x-j-4\pi t\xi)^2}{4t}}g_N^\pm(t,x)\eta^\pm(x)\chi(x-j-4\pi t\xi)dx$$
$$+\frac{e^{-i4\pi^2t\xi^2}}{2\sqrt{t}}\sum_{\pm, |j|<j_\epsilon}e^{i2\pi j\xi}\,A_j(t)\int_{-\infty}^\infty e^{i\frac{(x-j+4\pi t\xi)^2}{4t}}\overline{g_N^\pm(t,x)}\eta^\pm(x)\chi(x-j+4\pi t\xi)dx,$$
and $\chi(s)$ is a smooth function vanishing on $\{x,|x|< \frac 12\}$, and valued $1$ on $\{x,|x|>1\}$. In particular the support of $\chi'$ is bounded. Now we integrate by parts using the quadratic phases. Again since  $g_N^\pm (t,x)\eta^\pm(x)$ converge to zero at both $-\infty$ and $+\infty$ there are no boundary terms and we get: 
$$J_1^l(t,\xi)=-i\sqrt{t}e^{i4\pi^2t\xi^2}\sum_{\pm, |j|< j_\epsilon}e^{i2\pi j\xi}\,\overline{A_j(t)}\int_{-\infty}^\infty e^{-i\frac{(x-j-4\pi t\xi)^2}{4t}}\left(\frac{g_N^\pm(t,x)\eta^\pm(x)\chi(x-j-4\pi t\xi)}{x-j-4\pi t\xi}\right)_xdx$$
$$+i\sqrt{t}e^{-i4\pi^2t\xi^2}\sum_{\pm, |j|< j_\epsilon}e^{i2\pi j\xi}\,A_j(t)\int_{-\infty}^\infty  e^{i\frac{(x-j+4\pi t\xi)^2}{4t}}\left(\frac{\overline{g_N^\pm(t,x)}\eta^\pm(x)\chi(x-j+4\pi t\xi)}{x-j+4\pi t\xi}\right)_x dx.$$
When the derivative falls on $\chi$ or on the denominator, we get again smallness by using the dominated convergence theorem. We are left with the terms involving $(g_N^\pm)_x=-uT+i\frac{M}{\langle x\rangle}N^{\pm \infty} e^{-iM\log\frac {\langle x\rangle}{\sqrt{t}}}$. Now we note that we can discard also the last term of $(g_N^\pm)_x(t,x)$ \\ as for instance \footnote{Indeed, we can use for large $a$ the fact that $\int e^{is^2}\frac{e^{iM\log\langle a+s\rangle}}{s\langle a+s\rangle}\eta^\pm(a+s)\chi(s)ds=O(\frac 1a)+\int e^{is^2}\left(\frac{e^{iM\log\langle a+s\rangle}}{s^2\langle a+s\rangle}\eta^\pm(a+s)\chi(s)\right)_sds$, and split the integral into regions $\frac 12 \leq |s|\leq \frac a2, \frac a2\leq |s|\leq 2a,2a\leq |s|$ to get a $\frac 1a$-bound.}
$$\left|\int_{-\infty}^\infty e^{\mp i\frac{(x-j\mp4\pi t\xi)^2}{4t}}\frac{e^{\mp iM\log\frac{\langle x\rangle}{\sqrt{t}}}}{(x-j\mp 4\pi t\xi)\langle x\rangle}\eta^+(x)\chi(x-j\mp 4\pi t\xi)dx\right|\leq \frac{C}{|j\pm4\pi t\xi|}\leq \epsilon,$$
for $\xi>0$ far away from the finite set $\{j,|j|<j_\epsilon\}$ and choosing $\xi(\epsilon,t,\{\alpha_j\})$ larger if needed. Therefore, we are left with estimating the terms of $J_1^l(t,\xi)$ involving the $-uT$ part of $(g_N^\pm)_x$: there exists $\xi(\epsilon,t,\{\alpha_j\})$ such that for $\xi\geq  \xi(\epsilon,t,\{\alpha_j\})$ we have
$$|J^l(t,\xi)-J_2^l(t,\xi)|\leq C\epsilon,$$
with
$$J_2^l(t,\xi)=i\sum_{\pm,|j|< j_\epsilon}\,\overline{A_j(t)} \int_{-\infty}^\infty \frac{\sum_r A_r(t)e^{i\frac{x(j-r+4\pi t\xi)}{2t}}e^{-i\frac{j^2-r^2}{4t}}}{x-j-4\pi t\xi}T(t,x)\eta^\pm(x)\chi(x-j-4\pi t\xi)dx$$
$$-i\sum_{\pm,|j|< j_\epsilon}\,A_j(t)\int_{-\infty}^\infty  \frac{\sum_r \overline{A_r(t)}e^{i\frac{x(r-j+4\pi t\xi)}{2t}}e^{i\frac{j^2-r^2}{4t}}}{x-j+4\pi t\xi}T(t,x)\eta^\pm(x)\chi(x-j+4\pi t\xi)dx.$$
%We first get decay in $\xi$ by performing an integration by parts from $e^{\pm ix4\pi t\xi}$. When the derivative falls on $\eta^\pm(x)$ or the denominator $\frac 1{x-j\mp4\pi t\xi}$ we get an uniform bound on the integral. When the derivative falls on $T(t,x)$ it generates a quadratic phase, from which we can integrate by parts, after removing a bounded piece of the integral centered where the phase vanishes, to get again an uniform bound for the integral. We are thus left with estimating
%$$J_3^l(t,\xi)=-\frac{1}{4\pi t\xi}\sum_{|j|< j_\epsilon,r}\,e^{i(|\alpha_j|^2-|\alpha_r|^2)\log \sqrt{t}}\overline{A_j(t)}A_r(t) (j-r)\int_{-\infty}^\infty \frac{e^{i\frac{x(j-r+4\pi t\xi)}{2t}}}{x-j-4\pi t\xi}T(t,x)\eta^+(x)dx$$$$-\frac{1}{4\pi t\xi}\sum_{|j|< j_\epsilon,r}\,e^{-i(|\alpha_j|^2-|\alpha_r|^2)\log \sqrt{t}}A_j(t)\overline{A_r(t)}(r-j)\int_{-\infty}^\infty  \frac{e^{i\frac{x(r-j+4\pi t\xi)}{2t}}}{x-j+4\pi t\xi}T(t,x)\eta^+(x)dx.$$
%As from Lemma \ref{lemmalimT} we have that $|T(t,x)-T^\infty|$
Note that the summation $\sum_\pm$ and $\eta^\pm$ can be now removed as $\eta^++\eta^-=1$. 

We treat first the terms involving $|r|<j_\epsilon$. If needed we choose $\xi(\epsilon,t,\{\alpha_j\})$ larger such that for $|r|<j_\epsilon$ and $\xi\geq \xi(\epsilon,t,\{\alpha_j\})$ we have:
$$\frac 1{\pm(j-r)+4\pi t\xi}\leq \epsilon.$$
We perform in the corresponding integrals an integration by parts using the linear phase in $x$. Then, we get the $\epsilon$-smallness from the above constraint, and the integral that yields is uniformly bounded. Indeed, when the derivative falls either on $\chi,\eta^\pm$ or on the denominator $\frac 1{x-j\mp4\pi t\xi}$ we get immediately an uniform bound on the integral. When the derivative falls on $T(t,x)$ it generates a quadratic phase. Hence we can first remove a bounded piece of the integral centered where the phase vanishes, and then we can integrate by parts to get again a uniform bound for the integral. 

We are thus left with estimating the terms involving $|r|\geq j_\epsilon$, for which the linear phase might approach zero:  there exists $\xi(\epsilon,t,\{\alpha_j\})$ such that for $\xi\geq  \xi(\epsilon,t,\{\alpha_j\})$ we have
$$|J^l(t,\xi)-J_3^l(t,\xi)|\leq C\epsilon,$$
with
\begin{equation}\label{princl}
J_3^l(t,\xi)=i\sum_{|j|< j_\epsilon,|r|\geq j_\epsilon}\,\overline{A_j(t)}A_r(t)\, I^+(t,\xi,j,r)-i\sum_{|j|< j_\epsilon,|r|\geq j_\epsilon}\,A_j(t)\overline{A_r(t)}\, I^-(t,\xi,j,r),
\end{equation}
where
\begin{equation}\label{intosc}
I^\pm(t,\xi,j,r):=e^{\mp i\frac{j^2-r^2}{4t}}\int_{-\infty}^\infty \frac{e^{ix(\pm\frac{ j- r}{2t}+ 2\pi \xi)}}{x-j\mp 4\pi t\xi}\,T(t,x)\,\chi(x-j\mp 4\pi t\xi)\,dx.
\end{equation}

We first note that in view of \eqref{tailA} we have $\epsilon-$smallness of $\sum_{|r|\geq j_\epsilon}|A_r(t)|$. 
For $4\pi t\xi\notin\mathbb Z$ we can integrate by parts in $I^\pm(t,\xi,j,r)$ using the linear phase to get the  bound $\frac C{d(4\pi \xi,\frac{\mathbb Z}{t})}$. Therefore, we cannot control this way the $L^2(k,k+1)$ norm in $\xi$. To overcome this difficulty we shall prove that for $4\pi t\xi\notin\mathbb Z$:
\begin{equation}\label{boundintosc}|I^\pm(t,\xi,j,r)|\leq\left\{\begin{array}{c}C, \mbox{ if }|\pm\frac{j-r}{t}+4\pi \xi|\geq 1,\vspace{2mm}\\ C|\log (|\pm\frac{j-r}{t}+4\pi \xi|)|, \mbox{ if }|\pm\frac{j-r}{t}+4\pi \xi|<  1.\end{array} \right.\end{equation}
These bounds imply
$$|I^+(t,\xi,j,r)|+|I^-(t,\xi,j,r)|\leq\left\{\begin{array}{c}C, \mbox{ if }|\frac{j-r}{t}+4\pi \xi|\geq 1,|\frac{j-r}{t}+4\pi \xi|\geq 1,\vspace{2mm}\\ C |\log (|\frac{j-r}{t}+4\pi \xi|)|, \mbox{ if }|\frac{j-r}{t}+4\pi \xi|<  1,\vspace{2mm}\\ C |\log (|-\frac{j-r}{t}+4\pi \xi|)|, \mbox{ if }|-\frac{j-r}{t}+4\pi \xi|<  1.\end{array} \right.$$
Note that for $0<t<1$ the last two regions intersect if and only if $|2(j-r)|<2t<2$, that is when $r=j$ and in that case the bound is the same, $C\log(4\pi |\xi|)$. 
Then, by summing in $j$ and $r$, and by using \eqref{tailA}  we get for  $4\pi t\xi\notin\mathbb Z$ the bounds
    $$|J^l_3(t,\xi)|\leq\left\{\begin{array}{c}C\epsilon, \mbox{ if }d(4\pi \xi,\frac{\mathbb Z}{t})\geq 1,\vspace{2mm}\\ C\epsilon \,|\log (d(4\pi \xi,\frac{\mathbb Z}{t}))|, \mbox{ if }d(4\pi \xi,\frac{\mathbb Z}{t})<  1,\end{array} \right.$$
thus the lemma follows from \eqref{princl}.\\

We are thus left with proving \eqref{boundintosc}. We split the integral in $I^+(t,\xi,j,r)$ into the regions $x<0$ and $x>0$:
%. As the integrants are bounded, to get \eqref{boundintosc} and the similar estimate for $I^-(t,\xi,j,r)$ it is enough to get the desired bounds for
$$I^\pm(t,\xi,j,r)=e^{\mp i\frac{j^2-r^2}{4t}}\int_0^\infty \frac{e^{ix(\pm\frac{ j- r}{2t}+ 2\pi \xi)}}{x-j\mp 4\pi t\xi}\,T(t,x)\,\chi(x-j\mp 4\pi t\xi)\,dx$$
$$+e^{\mp i\frac{j^2-r^2}{4t}}\int_{-\infty}^0\frac{e^{ix(\pm\frac{ j- r}{2t}+ 2\pi \xi)}}{x-j\mp 4\pi t\xi}\,T(t,x)\,\chi(x-j\mp 4\pi t\xi)\,dx.$$
%$$=:I^1(t,\xi,j,r)+I^2(t,\xi,j,r).$$
By using the convergence rate in Lemma \ref{lemmalimT}:
$$|(T(t,x)-T^\infty)\mathbb{I}_{(0,\infty)}(x)|+|(T(t,x)-T^{-\infty})\mathbb{I}_{(-\infty,0)}(x)|\leq \frac C{\langle x\rangle},\forall x\in\mathbb R,$$
and in view of the definition of $\chi$ we get
$$|I^\pm(t,\xi,j,r)-\tilde I^\pm(t,\xi,j,r)|\leq C,$$
where
\begin{equation}\label{intoscprinc}
\tilde I^\pm(t,\xi,j,r):=T^{\infty}e^{\mp i\frac{j^2-r^2}{4t}}\int_{x>0,|x-j\mp4\pi t\xi|>1} \frac{e^{ix(\pm\frac{ j- r}{2t}+ 2\pi \xi)}}{x-j\mp 4\pi t\xi}\,dx
\end{equation}
$$+T^{-\infty}e^{\mp i\frac{j^2-r^2}{4t}}\int_{x<0,|x-j\mp4\pi t\xi|>1}\frac{e^{ix(\pm\frac{ j- r}{2t}+ 2\pi \xi)}}{x-j\mp 4\pi t\xi}\,dx.$$

If $|\pm\frac{ j- r}{t}+ 4\pi \xi|\geq 1$ we perform an integration by parts using the linear phase and get the bound uniform in $\xi,j,$ and $r$ in \eqref{boundintosc}. 

If $|\pm\frac{ j- r}{t}+ 4\pi \xi|< 1$ we denote for simplicity $a=-j\mp 4\pi t\xi$ and $b=\pm\frac{ j- r}{2t}+ 2\pi \xi$. We change variables $x+a=y$, $yb=s$ to rewrite $\tilde I^\pm(t,\xi,j,r)$ as:
\begin{equation}\label{mainosc}
e^{\mp i\frac{j^2-r^2}{4t}}e^{-iab}\,(T^{\infty}\int_{\frac s{b}>a,\,|s|>|b|}\frac{e^{is}}{s}\,ds+T^{-\infty} \int_{\frac s{b}<a,\,|s|>|b|}\frac{e^{is}}{s}\,ds),
\end{equation}
On the region where $|s|>1$, due to the oscillatory phase we get a bound uniform in $\xi,j,$ and $r$. 
Finally, on the region where $|b|<|s|<1$, if such regions exist, the integration of $\frac{e^{is}}{s}$ yields a $\log(|b|)$ bound.  Therefore we have obtained \eqref{boundintosc} and the Lemma follows.\\

For further purposes we note that we have obtained for $|\pm\frac{ j- r}{t}+ 4\pi \xi|< 1$ the estimate
\begin{equation}\label{estlog}
|\tilde I^\pm(t,\xi,j,r)-e^{\mp i\frac{j^2-r^2}{4t}}e^{-i(-j\mp 4\pi t\xi)(\pm\frac{ j- r}{2t}+ 2\pi \xi)}\,(T^{\infty}-T^{-\infty})\end{equation}
$$\times \int_{s>(-j\mp 4\pi t\xi)(\pm\frac{ j- r}{2t}+ 2\pi \xi),\,1>|s|>|\pm\frac{ j- r}{2t}+ 2\pi \xi|}\frac{e^{is}}{s}\,ds|\leq C,$$
with $C$ an universal constant.\end{proof}

\begin{lemma}\label{lemma_high}
There exists $\xi(\epsilon,t,\{\alpha_j\})$ such that for $\xi\geq  \xi(\epsilon,t,\{\alpha_j\})$ and $4\pi t\xi\notin\mathbb Z$ we have the bounds
$$|J^h(t,\xi)|\leq\left\{\begin{array}{c}C\epsilon, \mbox{ if }d(4\pi \xi,\frac{\mathbb Z}{t})\geq 1,\vspace{2mm}\\ C\epsilon \,|\log (d(4\pi \xi,\frac{\mathbb Z}{t}))|, \mbox{ if }d(4\pi \xi,\frac{\mathbb Z}{t})<  1,\end{array} \right.$$
\end{lemma}

\begin{proof}
Recall that
$$J^h(t,\xi)=\frac{e^{i4\pi^2t\xi^2}}{2\sqrt{t}}\sum_{\pm, |j|\geq j_\epsilon}e^{i2\pi j\xi}\,\overline{A_j(t)}\int_{-\infty}^\infty e^{-i\frac{(x-j-4\pi t\xi)^2}{4t}}g_N^\pm(t,x)\eta^\pm(x)dx$$
$$+\frac{e^{-i4\pi^2t\xi^2}}{2\sqrt{t}}\sum_{\pm, |j|\geq j_\epsilon}e^{i2\pi j\xi}\,A_j(t)\int_{-\infty}^\infty e^{i\frac{(x-j+4\pi t\xi)^2}{4t}}\overline{g_N^\pm(t,x)}\eta^\pm(x)dx.$$
%While smallness was often obtained for parts of $J^l(t,\xi)$ by using the constraint $|j|<j_\epsilon$ combined with the convergence $g_N^\pm(t,x)\overset{x\rightarrow\pm\infty}{\rightarrow}0$,
 In this case we will get the $\epsilon$-decay from $A_j(t)$ thanks to \eqref{tailA}. 
 %For doing so, it will be enough to bound the above integrals by a function of $\xi$ that belongs to $L^2(k,k+1)$. 
 We can remove the same pieces of the integrals as in the proof of the Lemma \ref{lemma_low} to end with %Just by working on the integral in $s$ we can obtain the above inequality \eqref{est3IBP} which is not good for $4\pi t\xi$ approaching $\mathbb Z$. To avoid this issue we shall combine both integrations in $\xi$ and in $s$.i\sum_{|j|< j_\epsilon,|r|\geq j_\epsilon}\,e^{i(|\alpha_j|^2-|\alpha_r|^2)\log \sqrt{t}}\overline{A_j(t)}A_r(t)\, I^+(t,\xi,j,r)$$
\begin{equation}\label{princh}
i\sum_{|j|\geq  j_\epsilon;r}\,\overline{A_j(t)}A_r(t)\, I^+(t,\xi,j,r)-i\sum_{|j|\geq  j_\epsilon;r}\,A_j(t)\overline{A_r(t)}\, I^-(t,\xi,j,r),
\end{equation}
where $I^\pm(t,\xi,j,r)$ were defined in \eqref{intosc}. 
This can be handled the same way as was done for $J^l_3(t,\xi)$ in the proof of Lemma \ref{lemma_low}.
 \end{proof}
 
%We note here that in view of \eqref{princl}, \eqref{princh} and the first estimate of \eqref{boundintosc} we have obtained  the existence of $\xi(\epsilon,t,\{\alpha_j\})$ such that for $\xi\geq  \xi(\epsilon,t,\{\alpha_j\})$ and $4\pi t\xi\notin\mathbb Z$:\begin{equation}\label{princJ}|J(t,\xi)-i\sum_{|\frac{ j- r}{t}+ 4\pi \xi|<1}\,\overline{A_j(t)}A_r(t)\, I^+(t,\xi,j,r)+i\sum_{|-\frac{ j- r}{t}+ 4\pi \xi|<1}\,A_j(t)\overline{A_r(t)}\, I^-(t,\xi,j,r)|\leq C.\end{equation}

 \begin{lemma}\label{lemma_princ}
 For any $\xi\in\mathbb R$ we have:
 $$|I(t,\xi)|\leq C\|\{A_j(t)\}\|_{l^1},$$
with $C$ an universal constant. Moreover, there exists $\xi(\epsilon,t,\{\alpha_j\})$ such that for $k\geq \xi(\epsilon,t,\{\alpha_j\})$
 $$\left|\int_k^{k+1}|I(t,\xi)|^2d\xi-4\pi \sum_j |\alpha_j|^2\right|\leq C\,\epsilon.$$
 \end{lemma}
 \begin{proof}
We start by performing some change of variables in the expression of $I(t,\xi)$:
$$I(t,\xi)=\frac 12\sum_\pm \int_{-\infty}^\infty e^{i2\pi x\xi}\,\sum_j\overline{A_j(t)}\frac{e^{-i\frac{(x-j)^2}{4t}}}{\sqrt{t}} N^{\pm \infty}  e^{-iM\log\frac  {\langle x\rangle}{\sqrt{t}}}\eta^\pm (x) dx$$
$$+\frac 12\sum_\pm\int_{-\infty}^\infty e^{i2\pi x\xi}\,\sum_jA_j(t)\frac{e^{i\frac{(x-j)^2}{4t}}}{\sqrt{t}} \overline{N^{\pm \infty}}  e^{iM\log\frac  {\langle x\rangle}{\sqrt{t}}}\eta^\pm (x) dx$$
$$=N^{\pm \infty}\frac{e^{i4\pi^2t\xi^2}}{2\sqrt{t}}\sum_{\pm, j}e^{i2\pi j\xi}\,\overline{A_j(t)}\int_{-\infty}^\infty e^{-i\frac{(x-j-4\pi t\xi)^2}{4t}}e^{-iM\log\frac  {\langle x\rangle}{\sqrt{t}}}\eta^\pm (x) dx$$
$$+ \overline{N^{\pm \infty}} \frac{e^{-i4\pi^2t\xi^2}}{2\sqrt{t}}\sum_{\pm,j}e^{i2\pi j\xi}\,A_j(t)\int_{-\infty}^\infty e^{i\frac{(x-j+4\pi t\xi)^2}{4t}}e^{iM\log\frac  {\langle x\rangle}{\sqrt{t}}}\eta^\pm (x) dx.$$
$$=N^{\pm \infty} e^{i4\pi^2t\xi^2}\sum_{\pm,j}e^{i2\pi j\xi}\,\overline{A_j(t)}\int_{-\infty}^\infty e^{-iy^2}e^{-iM\log\frac  {\langle 2\sqrt{t}y+j+4\pi t\xi\rangle}{\sqrt{t}}}\eta^\pm (2\sqrt{t}y+j+4\pi t\xi) dy$$
$$+ \overline{N^{\pm \infty}} e^{-i4\pi^2t\xi^2}\sum_{\pm,j}e^{i2\pi j\xi}\,A_j(t)\int_{-\infty}^\infty e^{iy^2}e^{iM\log\frac  {\langle 2\sqrt{t}y+j-4\pi t\xi\rangle}{\sqrt{t}}}\eta^\pm (2\sqrt{t}y+j-4\pi t\xi) dy.$$
We first note that the integrals are uniformly bounded in $j$ and $\xi$: the contribution of the bounded region $|y|<1$ is bounded as the integrant is of modulus less than one, while the contribution of the region $|y|>1$ is bounded by doing integrations by parts using the quadratic phase.  Therefore, we get the first bound of the Lemma. 

To estimate $\int_k^{k+1}|I(t,\xi)|^2d\xi$ we shall split $I(t,\xi)$ into a function of size of order $\epsilon$ and a function of $L^2(k,k+1)$-norm equal to $4\pi \sum_j |\alpha_j|^2$.
In view of the definition \eqref{tailA} of $j_\epsilon$, the terms in $I(t,\xi)$ involving $|j|> j_\epsilon$ can be upper-bounded by $C\epsilon$. We are left with the terms involving $|j|\leq j_\epsilon$. 
Observe that\footnote{Indeed, the integral $\int_{-\infty}^\infty e^{iy^2}(e^{iM\log\langle 2y+a\rangle}-e^{iM\log \langle a\rangle})\eta^\pm(2\sqrt{t}y+\sqrt{t}a)dy$ can be upper-bounded by $\frac C{|a|}$ on $|s|<1$, while on $|s|>1$ by performing integration by parts from the quadratic phase and by using the dominated convergence theorem we get decay to zero as $|a|\rightarrow\infty$. Then by the same type of arguments we have $\underset{a\rightarrow\pm\infty}{\lim}\left( \int_{-\infty}^\infty e^{iy^2}\eta^\pm(2\sqrt{t}y+\sqrt{t}a)dy-\int_{-\infty}^\infty e^{iy^2}dy\right)=0$ and $\underset{a\rightarrow\pm\infty}{\lim}\int_{-\infty}^\infty e^{iy^2}\eta^\mp(2\sqrt{t}y+\sqrt{t}a)dy=0$.
%$\int_{-\infty}^\infty se^{is^2}\frac{e^{iM\log\langle a+s\rangle}-e^{iM\log \langle a\rangle}}{s}\eta^\pm(a+s)ds|\leq \frac{C}{|a|}+2|\int_{|s|>1} (e^{is^2})_s\frac{e^{iM\log\langle a+s\rangle}-e^{iM\log \langle a\rangle}}{s}ds|\overset{|a|\rightarrow\infty}{\longrightarrow}0.$
}

$$\underset{a\rightarrow\pm\infty}{\lim}\left( \int_{-\infty}^\infty e^{iy^2}e^{iM\log\langle 2y+a\rangle}\eta^\pm(2\sqrt{t}y+\sqrt{t}a)dy-e^{iM\log \langle a\rangle}\int_{-\infty}^\infty e^{iy^2}dy\right)=0,$$
and
$$\underset{a\rightarrow\pm\infty}{\lim} \int_{-\infty}^\infty e^{iy^2}e^{iM\log\langle 2y+a\rangle}\eta^\mp(2\sqrt{t}y+\sqrt{t}a)dy=0.$$
Hence, choosing $\xi(\epsilon,t,\{\alpha_j\})$ larger if needed, for $|j|\leq j_\epsilon$ and $\xi\geq \xi(\epsilon,t,\{\alpha_j\})$ we get:
$$\left|\int_{-\infty}^\infty e^{\mp iy^2}e^{\mp iM\log\frac  {\langle 2\sqrt{t}y+j\pm 4\pi t\xi\rangle}{\sqrt{t}}}\eta^\pm(2\sqrt{t}y+j\pm 4\pi t\xi)ds-e^{\mp iM\log\frac{\langle j\pm 4\pi t\xi\rangle}{\sqrt{t}}}\int_{-\infty}^\infty e^{\mp is^2}ds\right|\leq\epsilon,$$
and
$$\left|\int_{-\infty}^\infty e^{\mp iy^2}e^{\mp iM\log\frac  {\langle 2\sqrt{t}y+j\pm 4\pi t\xi\rangle}{\sqrt{t}}}\eta^\mp(2\sqrt{t}y+j\pm 4\pi t\xi)ds\right|\leq\epsilon.$$
Therefore we have for $\xi\geq k_\epsilon$
\begin{equation}\label{decI1}
|I(t,\xi)-I^1(t,\xi)-I^2(t,\xi)|\leq C\,\epsilon,
\end{equation}
where 
$$I^1(t,\xi)=N^\infty e^{i4\pi^2t\xi^2}\sqrt{\pi}e^{i\frac{\pi}4}\sum_{|j|\leq j_\epsilon}e^{i2\pi j\xi}\,e^{-iM\log\frac{\langle j+ 4\pi t\xi\rangle}{\sqrt{t}}}\,\overline{A_j(t)} ,$$
and 
$$I^2(t,\xi)=\overline{N^{-\infty}} e^{-i4\pi^2t\xi^2}\sqrt{\pi}e^{-i\frac{\pi}4}\sum_{|j|\leq j_\epsilon}e^{i2\pi j\xi}\,e^{iM\log\frac{\langle j- 4\pi t\xi\rangle}{\sqrt{t}}}\,A_j(t).$$
Since $I_1^1(t,\xi)$ and $I_1^2(t,\xi)$ are uniformly bounded by $2\sqrt{\pi}\sum_j|A_j(t)|$, we obtain from \eqref{decI1}
$$\left|\int_k^{k+1}|I(t,\xi)|^2d\xi-\int_k^{k+1}|I^1(t,\xi)|^2d\xi-\int_k^{k+1}|I^2(t,\xi)|^2d\xi-\int_k^{k+1}I^1(t,\xi)\overline{I^2(t,\xi)}d\xi\right|\leq C\,\epsilon.$$
Then, as $|N^{\pm\infty}|=2$, Plancherel's formula gives us for $k\geq k_\epsilon$
$$\left|\int_k^{k+1}|I(t,\xi)|^2d\xi-4\pi \sum_{|j|\leq j_\epsilon}  |A_j(t)|^2-\int_k^{k+1}I^1(t,\xi)\overline{I^2(t,\xi)}d\xi\right|\leq C\, \epsilon.$$
Now we see that the crossed terms are
$$\int_k^{k+1}I^1(t,\xi)\overline{I^2(t,\xi)}d\xi=N^\infty.N^{-\infty} \,\pi e^{i\frac\pi 2}\sum_{|j_1|,|j_2|\leq j_\epsilon}\overline{A_{j_1}(t)}\overline{A_{j_2}(t)}\times$$
$$\times \int_k^{k+1}e^{i8\pi^2t\xi^2}e^{i2\pi (j_1-j_2)\xi}e^{-iM\log\frac{\langle j+ 4\pi t\xi\rangle}{\sqrt{t}}}e^{-iM\log\frac{\langle j- 4\pi t\xi\rangle}{\sqrt{t}}}d\xi.$$
One single integration by parts using the quadratic phase in $\xi$ gives us decay in $k$, so choosing $\xi(\epsilon,t,\{\alpha_j\})$ larger if needed we obtain for $k\geq \xi(\epsilon,t,\{\alpha_j\})$
$$\left|\int_k^{k+1}|I(t,\xi)|^2d\xi-4\pi \sum_{|j|\leq j_\epsilon}  |A_j(t)|^2\right|\leq C\,\epsilon.$$
Recalling the choice \eqref{tailA} of $j_\epsilon$ and the conservation law \eqref{consl2} we get for $k\geq \xi(\epsilon,t,\{\alpha_j\})$
$$ \left|\int_k^{k+1}|I(t,\xi)|^2d\xi-4\pi \sum_{j}  |\alpha_j|^2\right|\leq C\,\epsilon.$$
%Remains to deal with the case of integrating on $4\pi t\xi$ approaching $\mathbb Z$ for $|j|\geq j_\epsilon$.
\end{proof}

Summarizing we have decomposed
$$\widehat{T_x}(t,\xi)=:I(t,\xi)+J(t,\xi),$$
and proved in Lemmas \ref{lemma_low}-\ref{lemma_high}-\ref{lemma_princ} that there exists $\xi(\epsilon,t,\{\alpha_j\})\in\mathbb R$ such that for $\xi\geq  \xi(\epsilon,t,\{\alpha_j\})$ and $4\pi t\xi\notin\mathbb Z$ we have the bounds: 
$$|J(t,\xi)|\leq\left\{\begin{array}{c}C\epsilon, \mbox{ if }d(4\pi \xi,\frac{\mathbb Z}{t})\geq 1,\vspace{2mm}\\ C\epsilon\, |\log (d(4\pi \xi,\frac{\mathbb Z}{t}))|, \mbox{ if }d(4\pi \xi,\frac{\mathbb Z}{t})<  1;\end{array} \right.$$
$$|I(t,\xi)|\leq C,$$
and for all $k\geq \xi(\epsilon,t,\{\alpha_j\})$:
 $$\left|\int_k^{k+1}|I(t,\xi)|^2d\xi-4\pi \sum_j |\alpha_j|^2\right|\leq C\,\epsilon.$$
 We note that for $\xi$ in an interval of size one, there are only a finite number of possible locations where $d(4\pi\xi,\frac{\mathbb Z}{t})<  1$, depending only on $t$, and on these regions $J(t,\xi)$ is square integrable. 
Therefore
$$\left|\int_k^{k+1}|\widehat{T_x}(t,\xi)|^2d\xi-4\pi \sum_j |\alpha_j|^2\right|\leq C\,\epsilon, \forall k\geq \xi(\epsilon,t,\{\alpha_j\}).$$

The value of $0<\epsilon<1$ was arbitrary, the constant $C$ is independent of $\epsilon$, so we obtain the conservation law \eqref{cons}, and the proof of Theorem \ref{th} is complete.\\

\subsection{The result on the normal vectors}

In this subsection we obtain the results \eqref{consN} and \eqref{energyN0} from Remark \ref{Nrem}. We recall from Lemmas 4.5-4.7 in \cite{BV5} that we have a limit $\tilde N(0)$ at $t=0$ of
$$\tilde N(t,x)=e^{i\sum_{r\neq x}|\alpha_r|^2\log\frac{|x-r|}{\sqrt{t}}}N(t,x),$$
that is piecewise constant 
$$\tilde N(0,x)=\tilde N(0,x'), \forall x,x'\in(j,j+1),\forall j\in\mathbb Z,$$
and
$$\tilde N(0,j^\pm)=e^{i\sum_{r\neq j}|\alpha_r|^2\log|r-j|}e^{i Arg(\alpha_j)}\Theta_j(B^\pm_{|\alpha_j|}).$$
Here $B^\pm_{|\alpha_j|}\in\mathbb S^2+i\mathbb S^2$ are defined in \cite{GRV} in terms of the asymptotics at $\pm\infty$ of the normal vectors of the self-similar solution $\chi_{|\alpha_j|}$. 
It follows that at $t=0$ we have
$$\tilde N_x(0)=\sum_j(\tilde N(0,j^+)-\tilde N(0,j^-))\delta_j=\sum_j e^{i\sum_{r\neq j}|\alpha_r|^2\log|r-j|}e^{i Arg(\alpha_j)}\Theta_j(B_{|\alpha_j|}^+-B_{|\alpha_j|}^-)\delta_j,$$
so
$$\widehat{\tilde N_x}(0,\xi)=\sum_j e^{i\sum_{r\neq j}|\alpha_r|^2\log|r-j|}e^{i Arg(\alpha_j)}\Theta_j(B_{|\alpha_j|}^+-B_{|\alpha_j|}^-)e^{i2\pi j}.$$
As $\widehat{\tilde N_x}(0,\xi)$ is periodic in $\xi$, we get by Plancherel's theorem that for any $k$
$$\int_k^{k+1}|\widehat{\tilde N_x}(0,\xi)|^2d\xi=\sum_j |\Theta_j(B_{|\alpha_j|}^+-B_{|\alpha_j|}^-)|^2=\sum_j |B_{|\alpha_j|}^+-B_{|\alpha_j|}^-|^2.$$
Therefore, as we know from \cite{GRV} that 
$$ |B_{|\alpha_j|}^+-B_{|\alpha_j|}^-|^2=4|B_{|\alpha_j|,1}^+|^2=4(1-(A_{|\alpha_j|,1}^+)^2)=4(1-e^{-\pi |\alpha_j|^2}),$$ 
we obtain \eqref{energyN0}.\\

For $t>0$ we fix $\epsilon\in(0,1)$. %and treat first the unmodulated quantity $\|N(t)\|_{L^2_{sc}}$. 
We split:
$$\widehat{N_x}(t,\xi)=-\sum_\pm\int_{-\infty}^\infty e^{i2\pi x\xi}\,u(t,x) (T^{\pm\infty}+(T(t,x)-T^{\pm\infty}))\eta^\pm(x)dx$$
$$=-\int_{-\infty}^\infty e^{i2\pi x\xi}\,\sum_{\pm,j}A_j(t)\frac{e^{i\frac{(x-j)^2}{4t}}}{\sqrt{t}} (T^{\pm\infty}+(T(t,x)-T^{\pm\infty}))\eta^\pm(x)dx=:\tilde I(t,\xi)+\tilde J(t,\xi).$$
Proceeding as above for $J(t,\xi)$ we get the existence of $\xi(\epsilon,t,\{\alpha_j\})$ such that
$$\tilde J(t,\xi)=-\int_{-\infty}^\infty e^{i2\pi x\xi}\,\sum_{\pm,j}A_j(t)\frac{e^{i\frac{(x-j)^2}{4t}}}{\sqrt{t}} g_T^\pm(t,x)\eta^\pm(x)dx$$
$$=-e^{-i4\pi^2t\xi^2}\sum_{\pm,j}e^{i2\pi j\xi}\,A_j(t) \int_{-\infty}^\infty e^{i\frac{(x-j+4\pi t\xi)^2}{4t}}g_T^\pm(t,x)\eta^\pm(x)dx$$
satisfies, for $\xi\geq  \xi(\epsilon,t,\{\alpha_j\})$ and $4\pi t\xi\notin\mathbb Z$,
$$|\tilde J(t,\xi)|\leq\left\{\begin{array}{c}C\epsilon, \mbox{ if }d(4\pi \xi,\frac{\mathbb Z}{t})\geq 1,\vspace{2mm}\\ C\epsilon\, |\log (d(4\pi \xi,\frac{\mathbb Z}{t}))|, \mbox{ if }d(4\pi \xi,\frac{\mathbb Z}{t})<  1.\end{array} \right.$$
For $\tilde I(t,\xi)$ we make the changes of variable $x=j+2\sqrt{t}y$ and $s=y-2\pi\sqrt{t}\xi$:
$$\tilde I(t,\xi)=-2\sum_{\pm,j}T^{\pm\infty} e^{i2\pi j\xi}\,e^{-i|\alpha_j|^2\log \sqrt{t}}A_j(t)e^{-i\frac{j^2}{4t}}\int_{-\infty}^\infty e^{iy^2-i4\pi \sqrt{t}\xi y}\eta^\pm(j+2\sqrt{t}y)dy$$
$$=-2\sum_{\pm,j}T^{\pm\infty} e^{i2\pi j\xi}\,A_j(t) e^{-i\frac{j^2}{4t}}e^{-i4\pi^2t\xi^2}\int_{-\infty}^\infty e^{is^2}\eta^\pm(j+4\pi t\xi+2\sqrt{t}s)ds.$$
Since for $|j|>j_\epsilon$ we get $\epsilon-$smallness from the $A_j$'s, and in view of the definition of $\eta^\pm$, we have
$$|\tilde I(t,\xi)+2T^\infty e^{-i4\pi^2t\xi^2}\sum_je^{i2\pi j\xi}\,A_j(t) e^{-i\frac{j^2}{4t}}\sqrt\pi e^{i\frac\pi 4}|\leq C\epsilon.$$
In particular, we note that all the terms are uniformly bounded, so that by Plancherel's theorem we have
$$\left|\int_k^{k+1}|\tilde I(t,\xi)|^2d\xi-4\pi \sum_j |A_j(t) |^2\right|\leq C\epsilon.$$
Therefore, as in the case of the tangent vector $T$ we get that for $k\geq \xi(\epsilon,t,\{\alpha_j\})$
$$|\int_k^{k+1}|\widehat{N_x}(t,\xi)|^2d\xi-4\pi \sum_j |A_j(t) |^2|\leq C\epsilon.$$
As $\epsilon\in(0,1)$ was arbitrary we get \eqref{consN} by the conservation of mass \eqref{consl2}.

\begin{remark}
In view of the estimates we have obtained on the $J(t,\xi)$, it is natural to look for a logarithmic growth of $\hat T_x(t,\xi)$ in terms of the distance $d(4\pi \xi,\frac{\mathbb Z}{t})$. Moreover, the numerical computations given in \cite{DHV2} suggest the unboundedness of $\|\widehat T_x\|_\infty$ in the case of a regular planar polygon as initial data.\par
By doing similar computations to the ones in this section, and by using in particular \eqref{estlog}, we obtain for values of $\xi$ such that there exists $n\in\mathbb N$, $d\in(0,1)$ satisfying 
$$4\pi\xi=\frac nt+d,$$ 
the estimate:
\begin{equation}\label{T^nearZ/tbis}
\left|\hat{T_x}(t,\xi)-i\sum_{j}\,\overline{A_j(t)}A_{j+n}(t)\,e^{-i\frac{j^2-(j+n)^2}{4t}}\,(T^{\infty}-T^{-\infty})\right.
\end{equation}
$$\left.\times e^{ij\frac d2}\left(e^{in\frac d2}\int_{s>(-j- n)\frac d2,\,1>|s|>\frac d2}\frac{e^{is}}{s}\,ds-\int_{s>-j\frac d2,\,1>|s|>\frac d2}\frac{e^{is}}{s}\,ds\right)\right|\leq K(t,\{\alpha_j\}).$$

For instance, in the case of initial data $\alpha_0^n=\alpha_n^n =\delta$ and $\alpha_j^n=0$ for $j\notin\{0,n\}$, that corresponds to a polygonal line with two corners separated by a distance of size $n$, in \eqref{T^nearZ/tbis} the sum reduces to the case $j=0$, and we get:
\begin{equation}\label{T^nearZ/t42corners}
\left|\hat{T_x}(t,\xi)-i\overline{A_0(t)}A_n(t)\, e^{i\frac{n^2}{4t}}(T^{\infty}-T^{-\infty})\,\left(e^{in\frac d2}-1\right)e^{i\frac d2}\log \frac d2\right|\leq K(t,\{\alpha_j^n\}).
\end{equation}
For $d\ll \frac 1n$ the factor $e^{in\frac d2}-1$ ruins the $\log d$ growth.  Instead, for $d\approx \frac 1n$ we could look for a $\log n$ growth. Unfortunately, the results we have at hand about the IVP of \eqref{NLS} are not good enough, and we get a  constant $K(t,\{\alpha_j^n\})$ in $n$ that grows faster than $\log n$. On the other hand it seems rather natural to be able to solve \eqref{NLS} and the corresponding equation \eqref{BF} just under the condition that $\sum_j |\alpha_j|^2$ is finite. This question will be studied elsewhere.

\end{remark}

\section{An observation about the dynamics of a regular polygon}\label{sectanglesreg}
In this section we give some evidence that supports the conjecture made in \cite{DHV1} about the evolution of a regular planar polygon according to the binormal flow.
 
As recalled in the Introduction, the case when the initial curve in \eqref{BF} is a broken line with just one corner of angle $\theta$ located at $x=0$ was considered in  \cite{GRV}. In that paper the Hasimoto transformation is still used, and a solution is found considering as initial condition for \eqref{NLS}  $\alpha\delta_0$, where
$$\sin\frac\theta 2=e^{-\pi\frac{\alpha^2}{2}},$$
$u_\alpha(t,x)=\alpha\frac{e^{i\frac{x^2}{4t}}}{\sqrt{t}}$, and $a(t)=\frac {\alpha^2}t$.
As a consequence, and except in the trivial situation of one straight line where $\theta=\pi$, the filament function of the initial curve $\chi_\alpha(0)$, i.e. $\theta\delta_0$,   is not the limit of the filament functions of $\chi_\alpha(t)$. Nevertheless, it was proved in \cite{BV4} that this solution is unique and the corresponding initial value problem is well posed in an appropriate sense.

Similarly, if $\chi(0)$ is a broken line with several corners of angles $\theta_j$ located at the integers $x=j$ it was proved in \cite{BV5} that one has to consider the sequence $\{\alpha_j\}$ with modulus defined by
$$\sin\frac{\theta_j} 2=e^{-\pi\frac{|\alpha_j|^2}{2}}.$$
The phases are determined in a more complicated way involving the curvature and torsion angles of $\chi(0)$. Nevertheless,  if $\chi(0)$ is a planar polygon $\{\alpha_j\}$ can be taken real. Then we construct a solution of \eqref{NLS} with $a(t)=\frac{\sum_j|\alpha_j|^2}{t}$, and datum at time zero given by $\sum_j\alpha_j\delta_j$.

It is then natural to expect that in the case of a planar regular polygon with $N$ sides as initial data of \eqref{BF} one has to consider as initial data for \eqref{NLS}
\begin{equation}\label{id}\sum_j \alpha\,\delta_{\frac jN}\end{equation} 
with $\alpha>0$ defined by
$$\sin(\frac\pi N)=e^{-\pi\frac{\alpha^2}{2}}.$$
By using the Galilean invariance and assuming uniqueness, it was shown in \cite{DHV1} that the corresponding solution of \eqref{NLS} has to be written as
$$\psi(t,x)=\hat{\psi}(t,0)\sum_j e^{it(2\pi Nj)^2+i(2\pi Nj)x}.$$
In view of \eqref{id} and the Poisson summation formula $\sum_je^{i(2\pi Nj)x}=\frac 1N\sum_j\delta_\frac jN$, we have
$$\hat{\psi}(t,0)=\alpha N,$$
which therefore does not depend on time.

So, on one hand we have a behavior of the linear evolution
$$\psi(t,x)=\sum_j \hat{\psi}(t,0) e^{it\Delta}\delta_\frac jN,$$
and we can think that the conservation law \eqref{consl2} also holds in the periodic setting\footnote{We recall that the sequence $\{A_j(t)\}_{j\in\mathbb N}$ was found by doing a fixed point argument on $\tilde A_k(t)$ for the equation (24) in \cite{BV5}: 
$$i\partial_t \tilde A_k(t)=f_k(t)-\frac 1{8\pi t}(|\tilde A_k(t)|^2-|\alpha_k|^2)\tilde A_k(t),$$
where
$$f_k(t)=\frac{1}{8\pi t}\sum_{(j_1,j_2,j_3)\in NR_k}e^{-i\frac{k^2-j_1^2+j_2^2-j_3^3}{4t}}e^{-i\frac{|\alpha_k|^2-|\alpha_{j_1}|^2+|\alpha_{j_2}|^2-|\alpha_{j_3}|^2}{4\pi}\log\sqrt{t}}\tilde A_{j_1}(t)\overline{\tilde A_{j_2}(t)}\tilde A_{j_3}(t),$$
with initial data $\tilde A_k(0)=\alpha_k$. 
In particular we remark that for $N\in\mathbb N$, $\{B_j(t)\}_{j\in\mathbb R}$ with $B_j(t):=\tilde A_{N+j}(t)$ solves also the equation. Therefore if the initial data satisfies $\alpha_{k+N}=\alpha_k$ for all $k$, and there is uniqueness of the solution, then we conclude that $\tilde A_{k+N}(t)=\tilde A_k(t)$ for all $k$ and $t$, so the periodic setting is preserved.}. \\

\noindent
As a consequence we would get
$$N (\alpha N)^2=N|\hat{\psi}(t,0)|^2.$$
On the other hand, it was proved in \cite{DHV1} making use again of the Poisson summation formula, that for rational  times $t_{p,q}$ the Talbot effect holds: if $q$ is odd
 $$\psi(t_{p,q},x)=\frac{\hat{\psi}(t_{p,q},0)}{Nq} \sum_l\sum_{m=0}^{q-1} G(p,q,m)\delta_{l+\frac m{Nq}}(x)=: \sum_l\sum_{m=0}^{q-1}\alpha_{l,m}\delta_{l+\frac m{Nq}}(x),$$
with
$$|\alpha_{l,m}|=\frac{|\hat{\psi}(t_{p,q},0)|}{N\sqrt{q}}.$$
Then
$$|\alpha_{l,m}|^2=\frac{|\hat{\psi}(t_{p,q},0)|^2}{N^2q},$$
so
$$e^{-\pi \frac{|\alpha_{l,m}|^2}{2}}=e^{-\pi\frac{|\hat{\psi}(t_{p,q},0)|^2}{2N^2q}}=(e^{-\pi\frac{\alpha^2}{2}})^\frac 1q,$$
therefore the angles $\theta_{p,q}$ of the skew polygon at time $t_{p,q}$ satisfy
$$\sin(\frac{\theta_{p,q}}2)=\sin(\frac\pi N)^\frac 1q,$$
that is precisely the value given in \cite{DHV1} and obtained from the numerical data. Similarly one can repeat the argument if $q$ is even.
\bigskip

{\bf{Acknowledgements:}}  This research is partially supported by the Institut Universitaire de France, by the French ANR project SingFlows, by ERCEA Advanced Grant 2014 669689 - HADE, by MEIC (Spain) projects Severo Ochoa  SEV-2017-0718, and PGC2018-1228 094522-B-I00, and by  Eusko Jaurlaritza project  IT1247-19 and BERC program.


\begin{thebibliography}{99999}
\bibitem{ArHa} 
R.J.~Arms and F.R.~Hama, 
Localized-induction concept on a curved vortex and motion of an elliptic vortex
ring,
{\it Phys. Fluids }{\bf 8} (1965), 553--560.

\bibitem{BV4}  V.~Banica and L.~Vega, 
The initial value problem for the binormal flow with rough data, 
{\it  Ann. Sci. \'Ec. Norm. Sup\'er.} {\bf 48} (2015), 1421--1453.

\bibitem{BV4note}  V.~Banica and L.~Vega, 
Singularity formation for the 1-D cubic NLS and the Schr\"odinger map on $\mathbb S^2$, 
{\it  Comm. Pure Appl. Anal.} {\bf 17} (2018), 1317--1329.

\bibitem{BV5}  
V.~Banica and L.~Vega, 
Evolution of polygonal lines by the binormal flow, 
{\it  Ann. PDE}, to appear.

 \bibitem{CaTi}
 A.J.~Callegari and L.~Ting, 
 Motion of a curved vortex filament with decaying vertical core and axial velocity, 
 {\it SIAM. J. Appl. Math.} {\bf 35} (1978), 148--175.
 
  
 \bibitem{CaKa}
 R.~Carles and T.~Kappeler, 
 Norm-inflation with infinite loss of regularity for periodic NLS equations in negative Sobolev spaces,
{\it Bull. Soc. Math. France} {\bf 145} (2017), 623--642.
 
  
\bibitem{CaWe} 
T.~Cazenave and F.B.~Weissler,
The Cauchy problem for the critical nonlinear Schr\"odinger equation in $H\sp s$, 
{\it Nonlinear Anal., Theory Methods Appl.} {\bf 14} (1990), 807--836.

\bibitem{Ch} 
M.~Christ, 
Power series solution of a nonlinear Schr\"odinger equation, 
Mathematical aspects of nonlinear dispersive equations, 
131-155, Ann. of Math. Stud., 163, Princeton Univ. Press, Princeton, NJ, 2007.
 
 \bibitem{CFM}
 P. ~Constantin, C. ~Fefferman and A. J.~Majda,
 Geometric constraints on potentially singular solutions for
              the {$3$}-{D} {E}uler equations,
{\it Comm. Partial Differential Equations}, {\bf 21} (1996), 559--571.

\bibitem{ChCoTa}
M.~Christ, J.~Colliander and T.~Tao,
Asymptotics, frequency modulation, and low regularity ill-posedness for canonical defocusing equations,
{\it Am. J. Math. } {\bf 125} (2003), 1235--1293.

\bibitem{DaR} 
L.S.~Da Rios,  
On the motion of an unbounded fluid with a vortex filament of any shape,  
{\it Rend. Circ. Mat. Palermo} {\bf 22} (1906), 117--135.




\bibitem{GiVe} 
J.~Ginibre and G.~Velo, 
On a class of Schr\"odinger equations. 
I. The Cauchy problem, general case, 
{\it{J. Funct. Anal.}} {\bf 32} (1979), 1--71. 

\bibitem{DHV1} 
F.~de la Hoz and L.~Vega, 
Vortex filament equation for a regular polygon,
 {\it  Nonlinearity }  {\bf 27} (2014), 3031--3057.
 
 
\bibitem{DHV2} 
F.~de la Hoz and L.~Vega, 
On the relationship between the one-corner problem and the
              {$M$}-corner problem for the vortex filament equation,
{\it J. Nonlinear Sci.},
{\bf 28} (2018),  {2275--2327}.

\bibitem{DHKV}
F.~De La Hoz, S.~Kumar and L~Vega, 
On the evolution of the vortex filament equation for regular M-polygons with nonzero torsion, 
{\it ArXiv} 1909.01029.

  \bibitem{GRV} S.~Guti\'errez, J.~Rivas and L.~Vega, 
  Formation of singularities and self-similar vortex motion under the localized induction approximation, 
  {\it Commun. PDE} {\bf 28} (2003) 927--968. 


\bibitem{GrDe}
F.F.~Grinstein and C.R.~DeVore, 
Dynamics of coherent structures and transition to turbulence in free square jets,
{\it Physics of Fluids} {\bf 8} (1996), 1237--1251.

\bibitem{Gr} 
A. Gr\"unrock, 
Bi- and trilinear Schr\"odinger estimates in one space dimension with applications to cubic NLS and DNLS, 
{\it Int. Math. Res. Not.} (2005), 2525--2558. 


\bibitem{Ha}
H.~Hasimoto,
A soliton in a vortex filament, 
 {\it J. Fluid Mech.} {\bf 51} (1972), 477--485.
 
 
 \bibitem{JeSe}
 R.~L.~Jerrard and C.~Seis, 
 On the vortex filament conjecture for Euler flows, 
{\it  Arch. Ration. Mech. Anal.} {\bf 224} (2017), 135--172.

\bibitem{JeSm2}
R.~L.~Jerrard and D.~Smets,
On the motion of a curve by its binormal curvature,
{\it J. Eur. Math. Soc.} {\bf 17} (2015), 1148--1515. 

  \bibitem{KPV}
C. Kenig, G. Ponce and L. Vega,
On the ill-posedness of some canonical non-linear dispersive equations, 
 {\it Duke Math. J.}
{\bf 106}  (2001), 716--633.


\bibitem{KiViZh} 
R.~Killip, M.~Visan and X.~Zhang,
Low regularity conservation laws for integrable PDE, 
{\it Geom. Funct. Anal.} {\bf 28} (2018), 1062--1090.


\bibitem{Ki}
N.~Kishimoto, 
Well-posedness of the Cauchy problem for the Korteweg-de Vries equation at the critical regularity, 
{\it Differential Integral Equations} {\bf 22} (2009), 447--464.

\bibitem{Kita}
N. Kita, 
Mode generating property of solutions to the nonlinear Schr\"odinger equations in one space dimension, Nonlinear dispersive equations, 
{\it GAKUTO Internat. Ser. Math. Sci. Appl., Gakkotosho, Tokyo} {\bf 26} (2006), 111--128.

\bibitem{KoTa}
H.~Koch and D.~Tataru,
Conserved energies for the cubic NLS in 1-d, 
{\it Duke Math. J.}, {\bf 167} (2018), 3207--3313.


 \bibitem{MuTaUkFu}
Y.~Murakami, H.~Takahashi, Y.~Ukita and S.~Fujiwara, 
On the vibration of a vortex filament, 
{\it Appl. Phys. Colloquium} (1937), 1--5.

\bibitem{Oh}
T.~Oh, 
A remark on norm inflation with general initial data for the cubic nonlinear Schr\"odinger equations in negative Sobolev spaces,
{\it  Funkcial. Ekvac.} {\bf 60} (2017), 259--277.

\bibitem{VaVe} 
A.~Vargas and L.~Vega, 
Global wellposedness of 1D cubic nonlinear Schr\"odinger equation for data with
infinity $L^2$ norm, 
{\it J. Math. Pures Appl.} {\bf 80} (2001), 1029--1044.

 
 \end{thebibliography}
\end{document}